\documentclass{tran-l}

\usepackage{graphicx}

\usepackage{ivyspackage}

\usepackage{amssymb}
\usepackage{mathtools}

\theoremstyle{plain}
\newtheorem{theorem}{Theorem}[section]
\newtheorem{corollary}[theorem]{Corollary}
\newtheorem{lemma}[theorem]{Lemma}
\newtheorem{proposition}[theorem]{Proposition}
\newtheorem{claim}[theorem]{Claim}

\theoremstyle{definition}
\newtheorem{definition}[theorem]{Definition}

\theoremstyle{remark}
\newtheorem{remark}[theorem]{Remark}

\numberwithin{equation}{section}

\newcommand{\BDOrd}{\calD}

\begin{document}

\title[Rectifiability of the Singular Strata]
{Rectifiability of the Singular Strata for Harmonic Maps to DM-Complexes}

\author{Ivy Stoner}
\address{
Department of Mathematics,
Brown University,
Providence, 02912,
USA
}
\email{ivy\_stoner@brown.edu}

\author{Yitong Sun}
\address{
Department of Mathematics,
Johns Hopkins University,
Baltimore, 21218,
USA
}
\email{ysun143@jh.edu}

\subjclass[2020]{Primary 58E20; Secondary 53C43}

\begin{abstract}
We prove a rectifiability theorem for harmonic maps into DM-complexes. More specifically, if $u$ is a harmonic map from a $n$-dimensional Riemannian domain to an $N$-dimensional NPC DM-complex $Y$, the $k$-th singular stratum $\calS^k(u)$ of the singular set $\calS(u)$ is countably $k$-rectifiable for all $k \in \{0,...,n-2\}.$ Our proof follows the framework in \cite{dm}. This extends the rectifiability result for the singular set of harmonic maps into a $F$-connected complex in \cite{dees} and generalizes the rectifiability theorem with respect the singular strata of harmonic maps into a $F$-connected complex in \cite{bd}.
\end{abstract}

\maketitle
\section{Introduction}
The theory of harmonic maps into singular spaces has played an important role in rigidity problems in geometry and geometric group theory. A fundamental example is the work of Gromov and Schoen \cite{gs} on harmonic maps into locally finite Euclidean buildings, where the regularity result of equivariant harmonic maps provides a geometric proof of $p$-adic superrigidity. See \cite{bdm} for the extension of this result to non-locally finite Euclidean buildings. Since buildings and the corresponding cell complexes are not smooth manifolds, we cannot directly apply the classical Bochner method on the whole domain. Thus, we need the regularity of harmonic maps to guarantee that the Bochner formula method is applicable. 

This strategy was further developed by Daskalopoulos, Mese, and Vdovina \cite{dmv} for locally finite hyperbolic buildings and, more generally, NPC complexes with branching differentiable manifold structure -- DM-complexes. Their rigidity result shows that, under the appropriate conditions on the domain and the representation, an equivariant harmonic map into such a DM-complex is non-branching and totally geodesic. In the case of hyperbolic buildings, this follows a superrigidity statement since the singular set is sufficiently small. The codimension-two estimate allows us to localize the analysis away from the singular set and apply a differential geometric approach on the regular set across the whole domain.

Thus, the size of the singular set is an essential ingredient in the rigidity problem. Once we know that the singular set is of Hausdorff codimension at least 2, it is natural to expect that there exists a rectifiable structure for the singular strata in the quantitative stratification for harmonic maps. Rectifiability is a stronger statement than a Hausdorff dimension estimate: it says that the $k$-th singular stratum is covered by countably many $k$-dimensional Lipschitz images. The goal of this paper is to establish such a rectifiability theorem for harmonic maps into locally finite NPC DM-complexes.

Let $u:\Omega \to Y$ be a harmonic map from a Riemannian domain into an $N$-dimensional NPC DM-complex $Y$. The singular set $\calS(u)$ is the set of points $x$ such that no neighborhood of $x$ is mapped by $u$ into a single differentiable manifold of $Y.$ Following the splitting structure of DM-complexes in \cite{dm}, we can locally decompose the singular set according to the maximal Euclidean factor. In particular, near a point $x$ where $u$ splits with an optimal $\RR^j$-factor, we write
\[
u=(V,v):B_{\sigma}(x) \to \RR^j \times Y^{N-j},
\]
where $Y^{N-j}$ is a lower dimensional $F$-connected complex. The singular stratum $\calS^k_j(u)$ associated with optimal $\RR^j$ splitting then consists of points $x$ where the singular component $v$ has no $(k+1)$-homogeneous tangent map near $x.$ We prove that the $k$-th singular stratum
\[
\calS^k(u):=\bigcup_{j=0}^{N-1}\calS^k_j(u)
\]
is countably $k$-rectifiable for every $k=0,...,n-2$.
\begin{theorem}\label{thm:mainrect}
    If $u:\Omega\subseteq\RR^n\to X$ is a harmonic map into an NPC DM-complex of dimension $N$, then $\calS^k(u)$ is countably $k$-rectifiable for all $k\in \{0,1,\ldots,n-2\}$.
\end{theorem}

The paper includes both local and global results. The local theorems concern the singular component $v:B_{\sigma}(x) \to Y^{N-j}.$ We first establish a covering theorem to give Minkowski-type control on the quantitative $k$-th singular stratum of $v.$ Then we combine the covering arguments with a mean-flatness estimate, showing the $k$-rectifiability holds for both the quantitative $k$-th singular stratum of $v$ and the $k$-th singular stratum $\calS_j^k(u)$ with respect to the optimal splitting factor $\RR^j$.  Technically, the proof of local rectifiability follows from the covering arguments first developed by \cite{nv1} and \cite{nv2} for harmonic and approximate harmonic maps. In addition, the proof applies some techniques refined by \cite{dmsv} for Q-valued functions. The global theorem is then obtained by reducing the $k$-th singular set of $u$ to these local singular strata $\calS_j^k(u)$. By using a countable cover of $\calS^k(u)$ and the local rectifiability result to the corresponding singular strata, we have the global rectifiability theorem.

One essential difficulty is the splitting problem in the hyperbolic setting. In the case of the $F$-connected complex in \cite{bd}, there is a natural local splitting of the complex. However, following the metric estimate in \cite{dm}, we only have the asymptotical product structure on DM-complexes. In particular, in a local splitting $u=(V,v),$ the components $V,v$ are not necessarily harmonic since the metric on the target is only asymptotically isometric to a product metric, from the metric estimate in \cite{dm}. Another difficulty is that we need to compare frequency quantities centered at different nearby points in the quantitative stratification argument, not only varying the scale at one fixed point. This type of estimate is needed in Lemma \ref{comparefreq} and later in the mean-flatness argument. 

Due to these two difficulties, the classical order is not strong enough for the quantitative estimate needed in the proof. For this reason, in Section \ref{section:smoothedorder}we introduce a smoothed order $\Ord_{\phi,v}$ with respect to the singular component $v$. The concept of smoothed order started in \cite{dmsv} for Q-valued harmonic functions and was then adapted by \cite{dees} to harmonic maps into $F$-connected complexes. We have a lot of new work here to modify this smoothed order to be applicable for our asymptotically harmonic map $v$. This smoothing allows us to work with weighted interior integrals instead of sharp boundary integrals, which enables us to compare frequency quantities when the center point varies. The error term in the smoothed order absorbs the fact that $v$ is only asymptotically harmonic. Thus the smoothed order provides the frequency pinching estimates needed to prove Lemma \ref{comparefreq} and the mean-flatness bounds used for rectifiability.

This rectifiability result gives a finer description of the singular set than the previous codimension estimate needed for rigidity. The previous dimension bound shows that the singular set is small and the rectifiability result shows that it is geometrically organized that each $k$-stratum is contained in countably many $k$-dimensional Lipschitz submanifolds. Thus the singular set is not only small in measure, but also has an approximate $k$-dimensional structure. This is particularly significant for DM-complexes, where the singularities are from the branching of the target. Establishing rectifiability in the setting of DM-complexes shows that these branching singularities have a controllable geometric structure. Since locally finite hyperbolic buildings are one typical example of DM-complexes and motivate the rigidity results of \cite{dmv}, our theorem provides the corresponding rectifiability result for the singular strata appearing in the hyperbolic building rigidity.

\subsection{Outline of this paper}
\begin{itemize}
    \item In Section \ref{sec:preliminaries}, we collect the definitions and background results needed throughout the paper. 
    \item In Section \ref{sec:srata}, we introduce concepts including the singular strata, almost homogeneity, and quantitative strata in the DM-complex setting for following quantitative arguments. 
    \item In Section \ref{section:smoothedorder}, we develop the smoothed order for the singular component $v$ arising from a local splitting of $u=(V,v)$ and corresponding frequency pinching. We also prove related results needed for mean-flatness arguments. 
    \item In Section \ref{sec:k-hom}, we prove the $k$-homogeneity and pinching results needed for the quantitative stratification. We show that in certain balls, the singular set lies close to an affine $k$-space.
    \item In Section \ref{sec:inhomogeneities}, we prove the inhomogeneity translating estimates. These lemmas state that the non-$(k+1)$-homogeneity can be transferred between nearby points when the relevant frequency pinching is small.
    \item In Section \ref{sec:meanflatness}, we relate the smoothed frequency pinching to the mean flatness, which is an important tool in the covering arguments. We prove that if the singular component is not approximately $(k+1)$-homogeneous, we can control the mean flatness.
    \item In Section \ref{sec:mainresults}, we state the main results of the paper including local versions and global versions. All of these results are for singular points with no splitting. We also give the proof of Theorem \ref{thm:mainrect}.
    \item In Appendix, we provide the covering arguments used in the proof of the local estimates.
\end{itemize}

\section{Preliminaries}\label{sec:preliminaries}
We recall the notions of NPC spaces, DM-complexes, and harmonic maps into NPC targets. We also discuss blow-up maps and tangent maps for harmonic maps into DM-complexes, and introduce the relevant concepts of homogeneity. Finally, we define the smoothed height and energy that will be used to define the smoothed order later.
\subsection{NPC spaces}
A \textbf{non-positively curved (NPC) space} $(X,d)$ is a complete metric space satisfying following two conditions:

(i) It is a length space. For any two points $P,Q$ in $X,$ there exists a distance-minimizing geodesic connecting $P$ and $Q.$ 

(ii) The triangle comparison inequality holds. In particular, let points $P,Q,R\in X$, and choose geodesics $\gamma_{P,Q}$, $\gamma_{Q,R}$, and $\gamma_{R,P}$ joining the corresponding pairs of points. Denote their lengths by
\[
d(P,Q)=r,\qquad d(Q,R)=p,\qquad d(R,P)=q.
\]
For any $0\leq\lambda\leq1$, let $Q_\lambda$ be the point on the geodesic from $Q$ to $R$ which lies a fraction $\lambda$ of the way from $Q$ to $R$. Then
\[
d(Q_\lambda,Q)=\lambda p,\qquad d(Q_\lambda,R)=(1-\lambda)p.
\]
Consider the Euclidean comparison triangle with side lengths $p,q,r$, and denote its vertices by $\overline P,\overline Q,\overline R$, corresponding respectively to $P,Q,R$. Let
\[
\overline Q_\lambda=\overline Q+\lambda(\overline R-\overline Q)
\]
be the corresponding point on the Euclidean segment from $\overline Q$ to $\overline R$. The NPC condition requires that the distance from $P$ to $Q_\lambda$ in $X$ is no larger than the corresponding Euclidean distance from $\overline P$ to $\overline Q_\lambda$. Equivalently,
\[
d^2(P,Q_\lambda)
\leq
(1-\lambda)d^2(P,Q)
+\lambda d^2(P,R)
-\lambda(1-\lambda)d^2(Q,R).
\]
Intuitively, this means the triangle in NPC space is no thicker than the triangle in Euclidean space.

\subsection{DM-complexes and conical complexes} We define DM-complexes as in \cite[Section 2]{dm}.
\begin{definition}[Cell complex]
Let ${\bf E}^d$ be an affine space. A {\bf cell} is a convex piecewise linear
polyhedron $S\subset {\bf E}^d$ whose relative interior lies in some affine subspace
${\bf E}^i\subset {\bf E}^d$. We write $S^m$ for a cell of dimension $m$. A {\bf convex cell complex} $Y$ in ${\bf E}^d$ is a finite collection
$\mathcal F=\{S\}$ of cells such that:
\begin{enumerate}
    \item for each $S^m\in\mathcal F$, its boundary $\partial S^m$ is a union
    of $F^j\in\mathcal F$ with $j<m$, where $\{F^j\}$ are called the faces of $S^m$.
    \item if $F^j,S^m\in\mathcal F$ with $j<m$, and $F^j\cap S^m\neq\varnothing$,
    then $F^j\subset S^m$.
\end{enumerate}    
\end{definition}
One typical example of cell complexes is a simplicial complex whose cells are simplices.
\begin{definition}[Riemannian complex]
A {\bf Riemannian complex} is a cell complex $Y$ equipped with a collection of metrics $G=\{G^S\}$ such that 
\begin{enumerate}
    \item each cell $S$ of $Y$ is endowed with a smooth metric $G^S.$ 
    \item the metric $G^S$ is smooth in the interior of $S$ and extends smoothly to $\partial S.$
\end{enumerate}
Moreover, if $F$ is a face of $S$, then
\[
G^S|_{F}=G^{F}.
\]
\end{definition}

We assume that all cell complexes considered below have bounded cells. If unbounded cells are allowed, we will write {\bf unbounded cell complexes} explicitly. We also assume that
all cell complexes $Y$ are locally compact, Riemannian, and NPC with respect to the distance induced by $G$.

\begin{definition}[DM-complex]
Let $(Y,G)$ be a $k$-dimensional Riemannian complex. We say that $Y$ is endowed with a {\bf branching differentiable manifold structure}, or $Y$ is a {\bf DM-complex}, if for any two cells $S_1,S_2\subset Y$ with
$S_1\cap S_2\neq\varnothing$, there exists a complete $k$-dimensional smooth Riemannian manifold $M$ and an isometric totally geodesic embedding $i:M\to Y$ such that
\[
S_1\cup S_2\subset i(M).
\]
For simplicity, we identify $M$ with its image $i(M)$ and denote it as a {\bf DM}.
\end{definition}
If each DM in a DM-complex $(Y,G)$ is isometric to $n$- dimensional Euclidean space, then this DM-complex $(Y,G)$ is called an {\bf $F$-connected complex}. The NPC condition implies that the intersection of any two DM's in a DM-complex is totally geodesic in each of them. Moreover, instead of working with general DM-complexes, we will focus on a class of simpler complexes, which is called the {\bf conical complexes}.
\begin{definition}[Conical complex]
Let $(X,d)$ be an NPC space and let $O_X \in X$. For $\mu \in [0,1],$ define the mapping
\[
R_{O_X,\mu}:X \to X
\]
by letting $R_{O_X,\mu}(P)$ be the point lying a fraction $\mu$ of the way from $O_X$ to $P \in X$ along the unique geodesic joining them. We say that $X$ is {\bf conical with respect to $O_X$} if 
\begin{enumerate}
    \item for any $P,Q \in X$ and $\mu \in [0,1]$
            \[
                d\bigl(R_{O_X,\mu}(P), R_{O_X,\mu}(Q)\bigr) = \mu d(P,Q).
            \]
    \item $R_{O_X,\mu}$ is surjective for any $\mu \in (0,1].$ 
\end{enumerate}
The point $O_X$ is called the {\bf cone point} of $X.$
\end{definition}
\subsection{Harmonic maps into NPC spaces}
For a map $u: (\Omega,g) \to (X,d)$ where $\Omega$ is an $n$-dimensional Riemannian domain and $X$ is a metric space, the $\epsilon$-energy density function is defined in \cite[Section 1.2]{ks1} as
\[
e_{\epsilon}(x) = 
\begin{cases} 
\displaystyle\int_{y \in \partial B_{\epsilon}(x)} \frac{d^2(u(x), u(y))}{\epsilon^2} \frac{d\sigma}{\epsilon^{n-1}} & x \in \Omega_{\epsilon} \\ 
\mathmakebox[\widthof{$\displaystyle\int\limits_{y \in \partial B_{\epsilon}(x)} \frac{d^2(u(x), u(y))}{\epsilon^2} \frac{d\sigma}{\epsilon^{n-1}}$}][c]{0} &\text{otherwise}
\end{cases}
\]
where $d \sigma$ here is $(n-1)$- dimensional surface measure and $\Omega_{\epsilon}:=\{x \in \Omega: \dist(x,\partial \Omega) \geq \epsilon\}.$
Say $u$ has finite (unsmoothed) energy if 
\[ E^{u}:=\sup_{\varphi \in C_c(M), 0 \leq \varphi \leq 1} \limsup_{\epsilon \to 0} \int_{\Omega} \varphi e_{\epsilon} \, dvol_g < \infty.\]
From the result in \cite[Section 1.5]{ks1}, we know that as $\epsilon \to 0, \, e_{\epsilon}(x) \, dvol_g$ converges weakly to a Sobolev energy density measure $|du|^2(x) dvol_g$. This defines the energy formula in $\Omega$:
\[ E^{u}[\Omega]:=\int_{\Omega} |du|^2 dvol_g.\]
We say a continuous map $u:\Omega \to X$ is {\bf harmonic} if it is locally energy minimizing, i.e. for any $p \in \Omega,$ there exists $r>0$ such that the restriction map $u |_{B_r(p)}$ is the energy minimizer among all admissible maps in the space $W^{1,2}_u(B_r(p),X):=\{h \in W^{1,2}(B_r(p),X):d(u,h) \in W_0^{1,2}(B_r(p))\}$ (cf. \cite[Section 2.2]{ks1}). 

When $X$ is NPC, a nonconstant harmonic map $u: \Omega \rightarrow (X,d)$ has the following important monotonicity formula. Given  $x_0 \in \Omega$ and $r>0$ such that $B_{r}(x_0) \subset \Omega$, let the (unsmoothed) height and energy functions be defined as
\begin{eqnarray*} 
I^u(x_0,r) := \int_{\partial B_{r}(x_0)} d^2(u(x),u(x_0)) d\Sigma, \ E^u(x_0,r):=E^u[B_r(x_0)].
\end{eqnarray*}
When the point $x_0$ is fixed, we will sometimes use the notation $I^u(r):=E^u(x_0,r)$ and $E^u(r):=E^u(x_0,r)$. There exists a constant $c>0$ depending only on the $C^2$ norm of the domain metric $g$ (with $c=0$ when $g$ is the standard Euclidean metric) such that
\begin{equation}
r \mapsto e^{cr^2} \frac{r \ E^u(r) }{I^u(r)}, \ \ \  \ \ r \mapsto e^{cr^2}\frac{I^u(r) }{r^{n+1}} \ \ \ 
 \end{equation}
are non-decreasing.  
Recall that $I^u(r)>0$ for any $r>0,$ which follows from the fact that $d^2(u(x),u(x_0))$ is subharmonic (see \cite[Proposition 2.2]{gs} and the mean value inequality for subharmonic functions. As a non-increasing limit of continuous functions,
\[
\Ord^u(x_0):=\lim_{r \rightarrow 0}e^{cr^2} \frac{r \ E^u(r) }{I^u(r)}
\]
 is an upper semicontinuous function.  The value $\Ord^u(x_0)$ is called the (unsmoothed) order of $u$ at $x_0.$ In the following sections, we focus on the (unsmoothed) energy $E^v$, (unsmoothed) height $I^v$, and the (unsmoothed) order $\Ord^v$ with respect to the the asymptotically harmonic map $v.$

 \subsection{Blow-up maps and tangent maps}
Given a harmonic map $u:\Omega \to (Y,d_Y)$ into a DM-complex, let $x_0 \in \Omega$ with $\alpha:=\Ord^u(x_0)$ and $r_0>0$ such that $B_{r_0}(x_0) \subset \Omega$ and identify $(B_{r_0}(x_0),g)$ with the Euclidean ball $B_{r_0}(0) \subset \RR^n$ via normal coordinates centered at $x_0=0$. Let $u:(B_{r_0}(0),g) \to (Y,d_Y)$ be the restriction of $u$. We construct the {\bf blow-up map} $u_{\sigma}$ of $u$ at $x_0$ in the Gromov-Schoen sense: Define a function $\lambda^u:(0,r_0] \to (0,\infty)$ by 
\[
\lambda^u(\sigma)=\lambda_{\sigma}:=\left(\sigma^{1-n} I^u(0,\sigma)\right)^{\frac{1}{2}}.
\]
For $\sigma \in(0,r_0],$ the blow-up map of $u$ at $x_0$ is given by
\[
u_{\sigma}:(B_{\sigma^{-1}r_0}(0),g_{\sigma}) \to (Y, \lambda_{\sigma}^{-1}d_Y), \ u_{\sigma}(x)=u(\sigma x),
\]
where $g_{\sigma}(x):=g(\sigma x).$ Equivalently, we may define $u_{\sigma}$ on $B_1(0).$ 

The map $u_{\sigma}$ is also energy minimizing and satisfies $\Ord^{u_{\sigma}}(x,r)=\Ord^{u}(\sigma x,\sigma r).$ By construction, 
\[I^{u_{\sigma}}(1):=\int_{\partial B_1(0)} d_{\sigma}^2(u_{\sigma},u_{\sigma}(0)) \, d\Sigma=1.\]
By direct computation, for $\sigma>0$ sufficiently small
\[E^{u_{\sigma}}[B_1(0)] \leq 2 \Ord^{u_{\sigma}}(x_0)=2\Ord^{u}(x_0)=2\alpha,\]
i.e. $u_{\sigma}$ has uniformly bounded energy on $B_1(0).$ By \cite[Theorem 2.4.6]{ks1}, $u_{\sigma}$ is uniformly Lipschitz in any compact subset of $B_1(0).$ By \cite[Section 3]{gs}, for any sequence $\{sigma_i\}_{i=1}^{\infty}$ with $\sigma_i\to 0$ as $i\to \infty$, there exists an (unrelabeled) subsequence $\{u_i:=u_{\sigma_i}\}$ of blow-up maps converging uniformly to a harmonic map $u_*:B_1(0) \to (T_{u(0)}Y,d),$ where $T_{u(0)}Y$ is the tangent cone of $Y$ over $u(0).$ We call $u_*$ a {\bf tangent map to $u$ at $x_0$}. Furthermore, following \cite[Proposition 3.3]{gs}, $u_*$ is a non-constant intrinsically homogeneous map of order $\alpha,$ i.e. $u_*$ maps every ray from the origin in $B_1(0)$ onto a geodesic in $T_{u(0)Y}$ such that $d(u_*(tx),u_*(0))=t^{\alpha}d(u_*(x),u_*(0))$ for $t \geq 0.$

\subsection{Homogeneity}
When the target $X_C$ is actually a conical $F$-connected complex, we can define homogeneity extrinsically as in \cite[Section 2.5]{bd} by scaling with respect to the cone point.
\begin{definition}
    Let $X_C$ be a conical $F$-connected complex with cone point $O_X.$ A map $h:\RR^n \to X_C$ is {\bf homogeneous of degree $\alpha$ about $x_0$} if $h(x_0)=O_X$ and $h(x_0+tz)=t^{\alpha}h(x_0+z)$ holds for any $z \in \RR^n$ and $t>0.$
\end{definition}
\begin{definition}
We say a homogeneous degree $\alpha$ map $h:\RR^n \to X_C$ is {\bf $k$-homogeneous about the point $x_0$} if it is homogeneous about the point $x_0$ and there is a $k$-dimensional vector subspace $V \subseteq \RR^n$ such that $h(x+v)=h(x)$ for any $x \in \RR^n$ and $v \in V.$    
\end{definition}

\subsection{Smoothed height and energy}
In our case where the target space is a DM-complex, we define the {\bf smoothed order} for technical reasons. We will discuss this in detail in Section \ref{section:smoothedorder}. Here we first give the smoothed height and the smooth energy. Fix once for all $\phi:\mathbb{R}_{\geq 0} \to \mathbb{R}$ such that $\phi \equiv 1$ on $[0,1/2], \phi(x)=2-2x$ on $[1/2,1],$ and $\phi\equiv0$ otherwise. Then, for a map $u\in W^{1,2}(\Omega\subset\RR^n,Y_C^N)$, we define the smoothed height and the smooth energy as below:
\begin{align*}
E_{\phi,u}(x,r) &:=\int_{\RR^n} |\nabla u(y)|^2 \phi\left(\frac{|y-x|}{r}\right) \, dy, \\
 I_{\phi,u}(x,r)&:= -\int_{\RR^n} d^2(u(y),O_Y) |y-x|^{-1} \phi'\left(\frac{|y-x|}{r}\right) \, dy.
\end{align*}

\section{Splitting and singular strata}\label{sec:strata}
We begin by introducing the regular and singular sets, where the regular set consists of points whose image is contained in a single differentiable manifold of the target. We then define what it means for the ``splitting'' of a map $u$ at a point $x$ with an $\RR^j$-factor and introduce the optimal splitting data at a point. Using this splitting structure, we define the stratum $\calS^k_j(u)$ with respect to $\RR^j$-factor and the $k$-th singular stratum $\calS^k(u).$ This section ends with the definition of the quantitative strata.
\subsection{Local splitting and Metric estimates}
Let $w: \Omega \to (Y,G)$ be a map from a Riemannian domain into an $N$-dimensional DM-complex. We define the regular set $\calR(w)$ by a collection consisting of points $x \in \Omega$ such that there exists a ball $B_r(x)\subset \Omega$ and a differentiable manifold (DM) $M$ such that $w(B_r(x))\subset M$, and define the singular set by 
\[\calS(w):=\Omega\backslash \calR(w).\]
Fix $x_0 \in \Omega.$ We say that $w$ \textbf{splits on} $B_{\sigma}(x_0)\subset \Omega$ \textbf{with} $\RR^j$ \textbf{factor} if the restriction $w|_{B_{\sigma}(x_0)}$ admits a local representation in the sense of \cite[(16)--(17)]{dm},
\[
w=(V,v):B_{\sigma}(x_0) \longrightarrow \bigl( \RR^j \times Y^{N-j},d_G \bigr),
\]
where $Y^{N-j}$ is a lower dimensional NPC complex, $V(x_0)=0,$ and $v(x_0)=P_0 \in Y^{N-j}$. We call \[
(\sigma,V,v,j,Y^{N-j}) 
\]
the {\bf splitting data for $w$ at $x_0.$} In particular, the map $v$ is called the singular component of $w$.

The local model $\RR^j \times Y^{N-j}$ describes the product structure, but the metric $d_G$ is not necessarily a product metric. Let $F \subset Y^{N-j}$ be an $(N-j)$-dimensional flat of the local model containing $P_0.$ Then on $\RR^j \times F,$ we have
\[
(V,v)=(V^1,...,V^j,v^{j+1},...,v^N).
\]
Let $H(V):=G(V,0)$ on $\RR^j$ and let $h$ denote the metric on $F.$ Define the product metric 
\[
G_0(V,v):=H(V) \oplus h(v).
\]
Following \cite[Lemma 22]{dm} and \cite[Assumption 2]{dm}, in subsequent sections we assume that {\bf $G$ is asymptotically isometric to $G_0$ } as $v \to P_0.$

\subsection{Singular strata}
For $x\in \Omega$ and $j\in \{0,1,\ldots,N\}$, let \[\sigma_j(x):=\sup\{\sigma\geq 0: w \text{ splits  on }  B_{\sigma}(x) \text{ with } \RR^j \text{ factor}\}\] and let\[J(x):=\max\{j\in\{0,\ldots,N\}: \sigma_j(x)>0\}.\] 
    Thus, there exists $\sigma_{J(x)} >0$ such that \[
    w |_{B_{\sigma_{J(x)}}(x)}=(V,v):B_{\sigma_{J(x)}}(x) \to \RR^{J(x)} \times Y^{N-J(x)}.
    \]
    \begin{definition}
    We call $(\sigma_{J(x)},V,v,J(x),Y^{N-J(x)})$ {\bf optimal splitting data for $w$ at $x$}, and we say that $w$ {\bf splits at $x$ with optimal $\RR^{J(x)}$ factor}. 
    \end{definition}
    \begin{definition}
        For $w :\Omega \to X$ and $j \in \{0,...,N-1\},$ we define
         \[F_j(w):=\{x\in \calS(w): w \text{ splits at } x  \text{ with optimal } \RR^j \text{ factor with } j=J(x)\}.\]
    \end{definition}

    \begin{definition}
    Let $u=(V,v):B_{\sigma}(x)\to (\RR^j\times Y_C^{N-j},d_G)$ be a local representation as in \cite[(16)]{dm} of a harmonic map from Riemannian domain $\Omega \in \RR^n$ into a  $N$-dimensional conical DM-complex $Y_C.$ For $k \in \{0,...,n-2\},$ we define
   \begin{align*}
        \calS_0^k(v):=\{x\in F_0(v): \text{no tangent map to } v \text{ at } x \text{ is } (k+1)- \text{homogeneous}\}.
    \end{align*}
    For $j \in \{0,...,N\},$ we define
    \begin{align*}
        \calS_j^k(u):=\{x\in F_j(u): \text{for all local  representations } \ u=(V,v):B_{\sigma}(x)\to (\RR^j\times Y_C^{N-j},d_G)  \\  \text{ as  in \cite[(16)]{dm} with } \sigma\in (0,\sigma_j(x)], \ x\in \calS_0^k(v)\}.
    \end{align*}
    Finally, we define
    \begin{align*}
        \calS^k(u):=\bigcup_{j=0}^{N-1}\calS_j^k(u).
    \end{align*}
    \end{definition}
    \subsection{Almost homogeneity and quantitative singular strata} To prove the stratification arguments in the following sections, we need to introduce these concepts.
    \begin{definition}
    We say that a map $v:B_{\sigma}(x)\to Y_C$ from $B_{\sigma}(x)\subset \RR^n$ to a conical $F$-connected complex $Y_C$ is {\bf $(\eta,r,k)$-homogeneous at $y$} if for some $k$-homogeneous Lipschitz map $h:\RR^n\to Y_C$ (cf. \cite[Definition 2.8]{bd}), we have
    \begin{align*}
        \sup_{B_r(x)}d(v,h^x)\leq \eta\paren{r^{1-n}I_{\phi,v}(x,r)}^{1/2},
    \end{align*}
    where $h^x(y):=h(y-x)$ is defined to center the homogeneity of $h$ at $x$. 
    \end{definition}
    \begin{definition}
    For $\eta>0, r \in (0,1),$ we define
    \begin{align*}
        \calS_{0,\eta,r}^k(v)=\{x\in \calS_0^k(v): v \text{ is  not } (\eta,s,k+1)-\text{homogeneous at } x \text{ for  any } s \in [r,1]\}
    \end{align*}
    \begin{align*}
        \calS_{0,\eta}^k(v)=\{x\in \calS_0^k(v): v \text{ is not } (\eta,s,k+1)-\text{homogeneous at } x \text{ for any } s\in (0,1]\}
    \end{align*}
    \end{definition}
    Since the proof of the following lemma uses the blow-up estimates for the smoothed height developed in Section \ref{section:smoothedorder}, we will defer its proof to Section \ref{section:smoothedorder}.
    \begin{lemma}\label{bd3.7}
        $\calS_0^k(v)=\cup_{\eta>0}\calS^k_{0,\eta}(v)=\cup_{\eta>0}\cap_{r>0}\calS^k_{0,\eta,r}(v)$.
    \end{lemma}
    \begin{remark}
    Following the assumption in \cite{dm}, we know that the singular part $v$ of a harmonic map $u$ is an asymptotically harmonic map. Daskalopoulos and Mese showed in \cite{dm} that the singular part of a harmonic map into an NPC DM-complex also has tangent maps. Let $u=(V,v):\Omega\to (\RR^j \times Y_C^{N-j},d_G)$ be a local representation of a harmonic map into a conical DM-complex, and let $x_0\in \Omega$ be a point with $v(x_0)=P_0$. For any $\sigma>0$, define $v_{\sigma}:(B_{1}(0),g_{\sigma})\to (Y_C,\lambda_{\sigma}^{-1}d)$ by $v_{\sigma}(x)=v(\sigma x)$, where $d$ is the distance on $Y_C$, $g_{\sigma}(x)=g(\sigma x)$, and $\lambda_{\sigma}=(\sigma^{1-n}I^v(x_0,\sigma))^{1/2}$. Then, there exists a subsequence $\{v_{\sigma_i}\}$ converging locally uniformly to a nonconstant homogeneous harmonic map $v_*:B_1(0)\to T_{v(0)}Y_C$ as $\sigma_i \to 0$, called a {\bf tangent map to} $v$ \textbf{at} $x_0=0$.
    \end{remark}
\section{Smoothed order for singular part}\label{section:smoothedorder}
In this section, we define the smoothed order and the corresponding frequency pinching for the singular component $v$ of the local splitting $u=(V,v).$ The concept of smoothed order is originally introduced in \cite{dmsv}. Here, the smoothed order is further modified to adapt to the asymptotically harmonic map $v$ since $v$ is asymptotically harmonic due to the asymptotically product metric in \cite{dm}. We prove the corresponding monotonicity formulae and bounds on the smoothed order and height. Then we derive the scale and base-point estimates fitting in the following mean-flatness argument, ending with the convergence argument in energy densities in the blow-up setting.

Fix $j \in \{0,...,N-1\}$ and let $x_0 \in \calS^k_j(u).$ Then $J(x_0)=j.$ Let $(\sigma_j,V,v,j,Y_C^{N-j})$ be the optimal splitting data at $x_0.$ Identifying $x_0=0$ via normal coordinates, we have a local representation $u=(V,v):B_{\sigma_j}(0)\to (\RR^j\times Y_C^{N-j},d_G)$ of a harmonic map as in \cite[(17)]{dm}. By \cite[Proposition 53]{dm} there exists $C_1>0$ depending only on $\sigma_j$, the constant in \cite[(29)--(33)]{dm}, and the Lipschitz constant of $u$ on $B_{\sigma_j}(0)$ such that for all $x\in B_{\sigma_j/2}(0)$ and all $\sigma\in (0,\sigma_j/4)$, we have
    \begin{align*}
        (E^v)'(x,\sigma)&:=\der{\sigma}E^v(x,\sigma)\geq \frac{n-2}{\sigma}E^v(x,\sigma)+2\int_{\partial B_{\sigma}(x)}\abs{\pder[v]{r}}^2-C_1\paren{E^v(x,\sigma)+\frac{1}{\sigma}\int_{B_{\sigma}(x)}d^2(v,P_0)}.
    \end{align*}
    The authors claim this only for $x\in F_0(v)$, but the proof does not require this assumption. Applying \cite[Lemma 45]{dm} (the proof of which also does not require the assumption that $x\in F_0(v)$), there exists $C_1'>0$ depending only on $n$, $\sigma_j$, the constant in \cite[(29)--(33)]{dm}, and the Lipschitz constant of $u$ on $B_{\sigma_j}(0)$ such that for $x,\sigma$ as before, we have
    \begin{align*}
        (E^v)'(x,\sigma)&\geq \frac{n-2}{\sigma}E^v(x,\sigma)+2\int_{\partial B_{\sigma}(x)}\abs{\pder[v]{r}}^2-C_1'(E^v(x,\sigma)+I^v(x,\sigma)).
    \end{align*}
    Integrating this against $-\frac{\sigma}{r^2}\phi'\paren{\frac{\sigma}{r}}$ for $r\leq\sigma_j/4$, we get
    {\small\begin{align*}
        E_{\phi,v}'(x,r)&:=\der{r}E_{\phi,v}(x,r)\geq \int_0^{\infty}-\frac{\sigma}{r^2}\phi'\paren{\frac{\sigma}{r}}\paren{\frac{n-2}{\sigma}}E^v(x,\sigma)d\sigma+\int_0^{\infty}-\frac{2\sigma}{r^2}\phi'\paren{\frac{\sigma}{r}}\int_{\partial B_{\sigma}(x)}\abs{\pder[v]{r}}^2d\sigma\\
        &-C_1'\int_0^{\infty}-\frac{\sigma}{r^2}\phi'\paren{\frac{\sigma}{r}}(E^v(x,\sigma)+I(x,\sigma))d\sigma\\
        =&\frac{2\xi_{\phi,v}(r)}{r^2}+\frac{n-2}{r}\int_0^{\infty}-\frac{1}{r}\phi'\paren{\frac{\sigma}{r}}E^v(x,\sigma)d\sigma-C_1'\int_0^{\infty}-\frac{\sigma}{r^2}\phi'\paren{\frac{\sigma}{r}}E^v(x,\sigma)d\sigma\\
        &-C_1'\int_0^{\infty}-\frac{\sigma}{r^2}\phi'\paren{\frac{\sigma}{r}}I^v(x,\sigma)d\sigma,
    \end{align*}}
    where $\xi_{\phi,v}(r):=\int_0^{\infty}-\sigma\phi'(\frac{\sigma}{r})\int_{\partial B_{\sigma}(x)}\left|\frac{\partial v}{\partial r}\right|^2 \, d\sigma.$
    By using the fact that $\sigma/r\leq1$ whenever $\phi'(\sigma/r)>0$ on the third and fourth terms on the RHS, using the fact that $1/\sigma\leq 2/r$ whenever $\phi'(\sigma/r)>0$ on the fourth term, and integrating the second and third terms on the RHS by parts, we get
    \begin{equation}\label{eq:0}
        E_{\phi,v}'(x, r)\geq \paren{\frac{n-2}{r}-C_1'}E_{\phi,v}(x, r)+\frac{2}{r^2}\xi_{\phi,v}(r)-2C_1'I_{\phi,v}(x, r).
    \end{equation}
    By \cite[Proposition 37]{dm}, there exists $C_2>0$ depending only on $n$, $\sigma_j$, the Lipschitz constant of $u$ on $B_{\sigma_j}(0)$, and the constant in \cite[(29)--(33)]{dm}, such that for all $x\in B_{\sigma_j/2}(0)$ and $\sigma\in (0,\sigma_j/4)$, we have
    \begin{equation*}
        \int_{B_{\sigma}(x)}(2\eta|\nabla v|^2+\nabla \eta \cdot \nabla d^2(v,P_0))\leq C_2 \int_{B_{\sigma}(x)} \eta d^2(v,P_0)
    \end{equation*}
    and by choosing $\eta$ as an approximation to the characteristic function of $B_{\sigma}(x)$ and rearranging, we have
    \begin{equation*}
        \begin{split}
            2E^v(x, \sigma)&\leq \int_{\partial B_\sigma(x)}\pder{r}d^2(v,P_0)+C_2\int_{B_{\sigma}(x)}d^2(v,P_0)\\
            &\leq \frac{1-n}{\sigma}I^v(x,\sigma)+(I^v)'(x, \sigma)+C_2'\sigma(I^v(x, \sigma)+\sigma E^v(x, \sigma))
        \end{split}
    \end{equation*}
 for some $C_2'>0$, where the second inequality uses \cite[Lemma 45]{dm}. Note that $C_2'$ depends only on $n$, $\sigma_j$, the constant in \cite[(29)--(33)]{dm}, and the Lipschitz constant of $u$ on $B_{\sigma_j}(0)$. For $\sigma<R_0:=\min\{\frac{\sigma_j}{4},\sqrt{\frac{2}{C'_2}},\sqrt{\frac{n-1}{C'_2}}\}$, we have $(2-C_2'\sigma^2)E^v(\sigma)\geq 0$ and $(\frac{n-1}{\sigma}-C_2'\sigma)I^v(x,\sigma)\geq 0$, so for $\sigma<R_0$ we have 
 \[
 (I^v)'(x,\sigma) \geq (2-C'_2\sigma^2)E^v(x,\sigma)+\left(\frac{n-1}{\sigma}-C'_2\sigma\right)I^v(x,\sigma) \geq 0,
 \]
 which implies that for $\sigma<R_0$ we have
    \[
    \int_{B_{\sigma}(x)}d^2(v,P_0)=\int_0^{\sigma} I^v(x, r)dr \leq\int_0^{\sigma} I^v(x, \sigma)=\sigma I^v(x, \sigma)dr.
    \]
    Then,
    \begin{align*}
        2E^v(x,\sigma)&\leq \int_{\partial B_{\sigma}(x)}\pder{r}d^2(v,P_0)+C_2\sigma I^v(x,\sigma)\leq \int_{\partial B_{\sigma}(x)}\pder{r}d^2(v,P_0)+C_3 I^v(x,\sigma),
    \end{align*}
    where $C_3=C_2R_0$. Integrating this against $-\frac{1}{r}\phi'\paren{\sigma/r}$ for $r\leq R_0$ and using integration by parts on the LHS, we get
    \begin{align*}
        2E_{\phi,v}(x, r)\leq \int_{B_r(x)}-\frac{1}{r}\phi'\paren{\frac{|y-x|}{r}}\pder{r}d^2(v(y),P_0)dy+C_3\int_0^{\infty}-\frac{1}{r}\phi'\paren{\frac{\sigma}{r}}I^v(x,\sigma)d\sigma.
    \end{align*}
    Using the fact that $\frac{1}{r}\leq\frac{1}{\sigma}$ whenever $\phi'(\sigma/r)>0$ on the second term on the RHS, we obtain
    \begin{equation}\label{eq:1}
        \int_{B_r(x)}-\frac{1}{r}\phi'\paren{\frac{|y-x|}{r}}\pder{r}d^2(v(y),P_0)dy\geq 2E_{\phi,v}(x, r)-C_3I_{\phi,v}(x, r).
    \end{equation}
    By the same computations as in \cite[Lemma 4.1]{dees}, we have
    \begin{equation}\label{eq:4}
        I_{\phi,v}'(x, r)=\frac{n-1}{r}I_{\phi,v}(x, r)+\int_{B_r(x)}-\frac{1}{r}\phi'\paren{\frac{|y-x|}{r}}\pder{r}d^2(v(y),P_0)dy.
    \end{equation}
    Thus, by (\ref{eq:0}) and (\ref{eq:4}) we have
    \begin{align*}
        \der{r}\log \frac{rE_{\phi,v}(x, r)}{I_{\phi,v}(x, r)} =& \frac{1}{r}+\frac{\frac{d}{dr}E_{\phi,v}(x,r)}{E_{\phi,r}(x,r)}-\frac{\frac{d}{dr}I_{\phi,v}(x,r)}{I_{\phi,v}(x,r)}\\
        \geq& -C_1'\paren{1+2r\frac{I_{\phi,v}(x, r)}{rE_{\phi,v}(x, r)}}+\frac{2}{r^2E_{\phi,v}(x, r)}\biggl(\xi_{\phi,v}(r)\\
        &-\left.r\cdot\frac{rE_{\phi,v}(x, r)}{I_{\phi,v}(x, r)}\int_{B_r(x)}-\frac{1}{r}\phi'\paren{\frac{|y-x|}{r}}d(v(y),P_0)\pder{r}d(v(y),P_0)dy\right)\\
        =&-C_1'\paren{1+2r\frac{I_{\phi,v}(x, r)}{rE_{\phi,v}(x, r)}}+\frac{2}{r^2E_{\phi,v}(x, r)}\biggl(\xi_{\phi,v}(r)\\
        &-2r\cdot\frac{rE_{\phi,v}(x, r)}{I_{\phi,v}(x, r)}\int_{B_r(x)}-\frac{1}{r}\phi'\paren{\frac{|y-x|}{r}}d(v(y),P_0)\pder{r}d(v(y),P_0)dy\\
        &+\left.r\cdot\frac{rE_{\phi,v}(x, r)}{I_{\phi,v}(x, r)}\int_{B_r(x)}-\frac{1}{r}\phi'\paren{\frac{|y-x|}{r}}d(v(y),P_0)\pder{r}d(v(y),P_0)dy\right)\\
    \end{align*}
    Using \eqref{eq:1} on the final term on the RHS, we get
    {\small\begin{align*}
        \der{r}\log \frac{rE_{\phi,v}(x,r)}{I_{\phi,v}(x,r)}&\geq-C_1'-C_3-2C_1'r\frac{I_{\phi,v}(x,r)}{rE_{\phi,v}(x,r)}+\frac{2}{r^2E_{\phi,v}(x,r)}\biggl(\xi_{\phi,v}(r)\\
        &-2r\cdot\frac{rE_{\phi,v}(x,r)}{I_{\phi,v}(x,r)}\int_{B_r(x)}-\frac{1}{r}\phi'\paren{\frac{|y-x|}{r}}d(v(y),P_0)\pder{r}d(v(y),P_0)dy+\frac{r^2E^2_{\phi,v}(x,r)}{I_{\phi,v}(x,r)}\biggr)\\
        &=-C_1'-C_3-2C_1'r\frac{I_{\phi,v}(x,r)}{rE_{\phi,v}(x,r)}+\frac{2}{r^2E_{\phi,v}(x,r)}\int_{B_r(x)}-|y-x|^{-1}\phi'\paren{\frac{y}{r}}\biggl(|y-x|^2\abs{\pder[v]{r}}^2\\
        &-2|y-x|\frac{rE_{\phi,v}(x,r)}{I_{\phi,v}(x,r)}d(v(y),P_0)\pder{r}d(v,P_0)+\paren{\frac{rE_{\phi,v}(x,r)}{I_{\phi,v}(x,r)}}^2d^2(v(y),P_0)\biggr)dy
    \end{align*}}
    Since the target of $v$ is a conical $F$-connected complex, we may consider it to be embedded in $\RR^N$ for some $N$. In this situation, the previous inequality can be rewritten as
    \begin{align}
        \der{r}\log \frac{rE_{\phi,v}(x,r)}{I_{\phi,v}(x,r)}&\geq-C_1'-C_3-2C_1'r\frac{I_{\phi,v}(x,r)}{rE_{\phi,v}(x,r)}\notag\\
        &+\frac{2}{r^2E_{\phi,v}(x,r)}\int_{B_r(x)}-|y-x|^{-1}\phi'\paren{\frac{y}{r}}\abs{|y-x|\pder[v]{r}(y)-\frac{rE_{\phi,v}(x,r)}{I_{\phi,v}(x,r)}v(y)}^2dy\label{eq:2}\\
        &\geq -(C_1'+C_3)-2C_1'r\frac{I_{\phi,v}(x,r)}{rE_{\phi,v}(x,r)}\geq -A-2Ar\frac{I_{\phi,v}(x,r)}{rE_{\phi,v}(x,r)}\label{eq:3}
    \end{align}
    where $A:=C_1'+C_3$. Denote $\BDOrd_{\phi,v}(x,r):=\frac{rE_{\phi,v}(x,r)}{I_{\phi,v}(x,r)}$. Adding $A$ to both sides of \eqref{eq:3} and then multiplying both sides by $e^{Ar}\BDOrd_{\phi,v}(x,r)$, we have
    \begin{equation*}
        \der{r}e^{Ar}\BDOrd_{\phi,v}(x,r)\geq-2Are^{Ar}\geq\der{r}-Ar^2e^{Ar}.
    \end{equation*}
    The above inequality shows the following monotonicity property for the smoothed order:
    \begin{lemma}\label{monotoneorder}
        There exist $A,R_0>0$ depending only on $n$, $\sigma_j$, the constant in \cite[(29)--(33)]{dm}, and the Lipschitz constant of $u$ on $B_{\sigma_j}(0)$, such that for all $x\in B_{\sigma_j/2}(0)$, the functions 
        \[r\mapsto I_{\phi,v}(x,r), \qquad r\mapsto I^v(x,r),\] 
        and the \textbf{smoothed order} (or \textbf{smoothed frequency}) \textbf{function} 
        \[r\mapsto e^{Ar}(\BDOrd_{\phi,v}(x, r)+Ar^2):=\Ord_{\phi,v}(x,r)\] are nondecreasing for $r< R_0 \leq \sigma_j/4$. Thus, the \textbf{frequency pinching} \[W_s^r(x,v):=\Ord_{\phi,v}(x,r)-\Ord_{\phi,v}(x,s)\] is a nonnegative function on $B_{\sigma_j/2}(0)$ for $0<s\leq r\leq R_0 \leq \sigma_j/4$.
    \end{lemma}
    \begin{lemma}\label{monotoneheight}
        There exists $C(n)>0$ such that for all $y\in B_{\tau}(0)\subset B_{4\tau}(0)\subset B_{R_0}(0)$, we have $I_{\phi,v}(y,\tau)\leq CI_{\phi,v}(0,4\tau)$.
    \end{lemma}
    \begin{proof}
        Since $(I^v)'(0,r)\geq 0$, we have for all $r\in (2\tau,4\tau)$ that
        \begin{equation*}
            \int_{B_{2\tau(0)}}d^2(v,P_0)\leq C\tau\int_{\partial B_r(0)}d^2(v,P_0).
        \end{equation*}
        Integrating both sides against $-r^{-1}\phi'(r/4\tau)$, we see that
        \begin{equation}\label{intdbound}
            \int_{B_{2\tau}(0)}d^2(v,P_0)\leq C\tau I_{\phi,v}(0,4\tau),
        \end{equation}
        and since $B_{\tau}(y)\subset B_{2\tau}(0)$, we have $I_{\phi,v}(y,\tau)\leq C\int_{B_{2\tau}(0)}d^2(v,P_0)$.
    \end{proof}

    \begin{lemma}\label{ineq:ord}
        There exists $C>0$ depending only on $n$, $\sigma_j$, the Lipschitz constant of $u$ on $B_{\sigma_j}(0)$, and the estimates in \cite[(29)--(33)]{dm}, such that for all $y\in B_{\tau/4}(0)\subset B_{16\tau}(0)\subset B_{R_0}(0)$, we have $\Ord_{\phi,v}(y,\tau)\leq C(\Ord_{\phi,v}(0,16\tau)+1)$.
    \end{lemma}
    \begin{proof}
    We will use the following two estimates \eqref{eq:10} and \eqref{eq:11} for the smoothed heights in this proof. To make the proof easy to follow, we postpone the computations leading to \eqref{eq:10} and \eqref{eq:11} until after the proof of this lemma.
    
    For $\eta\in C_c^{\infty}(B_{\sigma_j}(0))$ with $0\leq\eta\leq 1$ and $t>0$ sufficiently small, define $v_{t\eta}:B_{\sigma_j}(0)\to (Y_C^{N-j},h)$ by $v_{t\eta}(x)=(1+t\eta(x))v(x)$ and $u_{t\eta}:B_{\sigma_j}(0)\to(\RR^j\times Y_C^{N-j},d_G)$ by $u_{t\eta}(x)=(V_{t\eta}(x),v_{t\eta}(x))$ for all $x\in \Omega$. Following the proof of \cite[Proposition 37]{dm}, we obtain
    \begin{equation}\label{eq:8}
        \limsup_{t\to 0^+}\frac{E(v)-E(v_{t\eta})}{t}\leq C_2\int_{B_{\sigma_j}(0)}\eta d^2(v,P_0).
    \end{equation}
    Since $(Y_C^{N-j},d_h)$ is a conical $F$-connected complex, we obtain as in the proof of \cite[Lemma 3.1]{dees},
    \begin{equation}\label{eq:9}
        \limsup_{t\to 0^+}\frac{E(v)-E(v_{t\eta})}{t}=\int_{B_{\sigma_j}(0)}-2\eta|\nabla v|^2-\nabla \eta\cdot\nabla d^2(v,P_0).
    \end{equation}
    For $x\in B_{\sigma_j/2}(0)$ and $\sigma\in (0,R_0)$, using (\ref{eq:8}) and (\ref{eq:9}) and taking $\eta$ to be an approximation to the characteristic function of $B_{R_0}(x)$ give
    \begin{equation}\label{eq:9.5}
        \int_{\partial B_{\sigma}(x)}\pder{r}d^2(v,P_0)-2E^v(x,\sigma)\leq C_2\sigma I^v(x,\sigma)\leq C_3I^v(x,\sigma)
    \end{equation}
    Integrating this against $-\frac{1}{r}\phi'(\sigma/r)$ for $r\in (0,R_0)$, using the fact that $\frac{1}{r}\leq \frac{1}{\sigma}$ whenever $\phi'(\sigma/r)>0$, and using \eqref{eq:4}, we have
    \begin{equation}\label{eq:10}
        \der{r}\log(r^{1-n}I_{\phi,v}(x,r))=\der{r}\log I_{\phi,v}(x,r)-\frac{n-1}{r}\leq 2\frac{\BDOrd_{\phi,v}(x,r)}{r}+C_3.
    \end{equation}
    On the other hand, by \eqref{eq:1}, we have
    \begin{equation}\label{eq:11}
        \der{r}\log(r^{1-n}I_{\phi,v}(x,r))= \der{r}\log I_{\phi,v}(x,r)-\frac{n-1}{r}\geq 2\frac{\BDOrd_{\phi,v}(x,r)}{r}-C_3.
    \end{equation}
    Thus, for $0<s<r<R_0$, we use (\ref{eq:10}) and have
    \begin{align}
        \log\frac{r^{1-n}I_{\phi,v}(x,r)}{s^{1-n}I_{\phi,v}(x,s)}&\leq\int_s^r\frac{2\BDOrd_{\phi,v}(x,t)}{t}+C_3dt\leq\int_s^r\frac{2\Ord_{\phi,v}(x,t)}{t}+C_3dt\notag\\
        &\leq\int_s^r\frac{2\Ord_{\phi,v}(x,r)}{s}+C_3dt\leq\frac{2r}{s}\Ord_{\phi,v}(x,r)+C_3R_0.\label{eqn:6}
    \end{align}
    On the other hand, we use (\ref{eq:11}) and also have
    \begin{align}
        \log\frac{r^{1-n}I_{\phi,v}(x,r)}{s^{1-n}I_{\phi,v}(x,s)}&\geq \int_s^r\frac{2\BDOrd_{\phi,v}(x,t)}{t}-C_3dt=\int_s^r\frac{2}{t}\paren{e^{-At}\Ord_{\phi,v}(x,t)-At^2}-C_3dt\notag\\
        &\geq \int_s^r\frac{2e^{-Ar}}{r}\Ord_{\phi,v}(x,s)-(2Ar+C_3)dt\notag\\
        &\geq\frac{2(r-s)}{r}e^{-AR_0}\Ord_{\phi,v}(x,s)-(2AR_0^2+C_3R_0).\label{heightbound2}
    \end{align}
    For all $\tau$ with $16\tau\leq R_0$ and all $y\in B_{\tau/4}(0)$, Lemma \ref{monotoneheight} implies that
    \begin{equation*}
        \frac{(4\tau)^{1-n}I_{\phi,v}(y,4\tau)}{\tau^{1-n}I_{\phi,v}(y,\tau)}\leq 16^{n-1}C^2\frac{(16\tau)^{1-n}I_{\phi,v}(0,16\tau)}{(\tau/4)^{1-n}I_{\phi,v}(0,\tau/4)}.
    \end{equation*}
    Taking $s=\tau$, $r=4\tau$ in \eqref{heightbound2} and $s=\tau/4$, $r=16\tau$ in \eqref{eqn:6} gives us
    \begin{align*}
        \frac{3e^{-AR_0}}{4}\Ord_{\phi,v}(y,\tau)-(2AR_0^2+C_3R_0)\leq 4(n-1)\log2+2\log C+C_3R_0+128\Ord_{\phi,v}(0,16\tau)
    \end{align*}
    Rearranging this, we get that there exists $C$ such that $\Ord_{\phi,v}(y,\tau)\leq C(\Ord_{\phi,v}(0,16\tau)+1)$.
\end{proof}
\begin{proof}[Computations of \eqref{eq:10} and \eqref{eq:11}]
Inequality \eqref{eq:9.5} is equivalent to 
    \[
    2E^v(x,\sigma) \geq \int_{\partial B_{\sigma}(x)} \frac{\partial}{\partial r}d^2(v,P_0)- C_3I^v(x,\sigma).
    \]
Integrating both sides against $-\frac{1}{r}\phi'(\sigma/r)$ for $r \in (0,R_0)$ and applying integration-by-parts, we have
\begin{align*}
    2 \int_0^{\infty} \phi\paren{\frac{\sigma}{r}} \frac{d}{d\sigma}E^v(x,\sigma)\, d\sigma = 2 E_{\phi,v}(x,r) \geq& \int_0^{\infty} -\frac{1}{r}\phi'\paren{\frac{\sigma}{r}}\int_{\partial B_{\sigma}(x)} \frac{\partial}{\partial r}d^2(v,P_0) \,d\Sigma\, d\sigma\\
    &-C_3 \int_0^{\infty} -\frac{1}{r}\phi'\paren{\frac{\sigma}{r}}I^v(x,\sigma)d\sigma\\
    &=\int_{\RR^n}-\frac{1}{r}\phi'\paren{\frac{|y-x|}{r}}\frac{\partial}{\partial r}d^2(v(y),P_0) \, dy\\ 
    &-C_3 \int_0^{\infty} -\frac{1}{r}\phi'\paren{\frac{\sigma}{r}}I^v(x,\sigma)d\sigma.
\end{align*}
Since $\frac{r}{2} \leq \sigma \leq r$ and then $\frac{1}{r} \leq \frac{1}{\sigma},$ we derive
\[
C_3 \int_0^{\infty} -\frac{1}{r}\phi'\paren{\frac{\sigma}{r}}I^v(x,\sigma)d\sigma \leq \int_0^{\infty}-\frac{1}{\sigma}\phi'\paren{\frac{\sigma}{r}}I^v(x,\sigma) \, d\sigma =I_{\phi,v}(x,r).
\]
Combining together, 
\[
2E_{\phi,v}(x,r) \geq \int_{\RR^n}-\frac{1}{r}\phi'\paren{\frac{|y-x|}{r}}\frac{\partial}{\partial r}d^2(v(y),P_0) \, dy -C_3 I_{\phi,v}(x,r).
\]
By \eqref{eq:4},
\[
I_{\phi,v}'(x,r)=\frac{n-1}{r}I_{\phi,v}(x,r)+\int_{\RR^n}-\frac{1}{r}\phi'\paren{\frac{|y-x|}{r}} \frac{\partial}{\partial r}d^2(v(y),P_0)dy.
\]
Thus, we derive \eqref{eq:10} as below:
\begin{eqnarray*}
\frac{d}{dr}\log(r^{1-n}I_{\phi,v}(x,r))&=&(I_{\phi,v}(x,r))^{-1}\int_{\RR^n}-\frac{1}{r}\phi'\paren{\frac{|y-x|}{r}} \frac{\partial}{\partial r}d^2(v(y),P_0)dy\\
&\leq& \frac{2E_{\phi,v}(x,r)+C_3I_{\phi,v}(x,r)}{I_{\phi,v}(x,r)}\\
&=&2\frac{\calD_{\phi,v}(x,r)}{r}+C_3.
\end{eqnarray*}
For \eqref{eq:11}, we rearrange \eqref{eq:1} and have
\[
(I_{\phi,v}(x,r))^{-1}\int_{\RR^n}-\frac{1}{r}\phi'\paren{\frac{|y-x|}{r}} \frac{\partial}{\partial r}d^2(v(y),P_0)dy \geq 2\frac{E_{\phi,v}(x,r)}{I_{\phi,v}(x,r)} - C_3.
\]
Applying \eqref{eq:4}, we have
\begin{align*}
\frac{d}{dr}\log(r^{1-n}I_{\phi,v}(x,r))&=\frac{1-n}{r}+\frac{\frac{d}{dr}I_{\phi,v}(x,r)}{I_{\phi,v}(x,r)}\\
&=(I_{\phi,v}(x,r))^{-1}\int_{\RR^n}-\frac{1}{r}\phi'\paren{\frac{|y-x|}{r}} \frac{\partial}{\partial r}d^2(v(y),P_0)dy.
\end{align*}
Thus,
\[
\frac{d}{dr}\log(r^{1-n}I_{\phi,v}(x,r)) \geq 2\frac{E_{\phi,v}(x,r)}{I_{\phi,v}(x,r)}-C_3=\frac{2\calD_{\phi,v}(x,r)}{r}-C_3.
\]
\end{proof}
Now we are ready to give the proof of Lemma \ref{bd3.7}.
\begin{proof}[Proof of Lemma \ref{bd3.7}]
    Suppose that there exists $x\in \calS_0^k(v)$ such that for all $\eta_1\geq \eta>0$ there exists $s\in (0,1]$ such that $v$ is $(\eta,s,k+1)$-homogeneous at $x$. Then, for all $i\in\NN$ there exists $s_i\in (0,1]$ and a $(k+1)$-homogeneous map $h_i:B_{s_i}(x)\to Y_C$ such that
    \begin{equation*}
        \sup_{B_{s_i}(x)}d(v,h_i)\leq \frac{1}{i}\paren{s_i^{1-n}I_{\phi,v}(x,s_i)}^{1/2}.
    \end{equation*}
    Suppose that $\inf_{i\in \NN}s_i=s>0$. Then, we have
    \begin{equation*}
        \sup_{B_s(x)}d(v,h_i)\leq \frac{1}{i}\paren{s_i^{1-n}I_{\phi,v}(x,s_i)}^{1/2}\leq \frac{1}{i}\paren{I_{\phi,v}(x,1)}^{1/2},
    \end{equation*}
    which implies that $h_i\to v$ uniformly as $i\to\infty$. By choosing a subsequence of the $h_i$'s, we see that there exists a $(k+1)$-dimensional subspace $V$ containing $x$ such that $v(x)=v(y)$ for all $y\in V$. This implies that any tangent map to $v$ at $x$ is $(k+1)$-homogeneous, contradicting $x\in \calS_0^k(v)$. On the other hand, suppose that for an unrelabeled subsequence $\{s_i\}$ we have $s_i\to 0$ as $i\to \infty$. Then, let $\tilde{v}_i,\tilde{h}_i:B_1(0)\to (Y_C,\paren{s_i^{1-n}I_{\phi,v}(x,s_i)}^{-1/2}d)$ be given by $\tilde{v}_i(y)=v(x+s_iy)$ and $\tilde{h}_i(y)=h_i(x+s_iy)$. Then, we have
    \[\sup_{B_1(0)}d(\tilde{v}_i,\tilde{h}_i)\leq1/i.\]
    Recall that $\frac{d}{dr}I^v(x,\sigma)\geq 0$ for $\sigma<R_0,$ so for $r<R_0,$
    \[
    I_{\phi,v}(x,r)=-\int_0^r\frac{1}{\sigma}\phi'(\frac{\sigma}{r})I^v(x,\sigma) \, d\sigma.
    \]
    Since $-\phi'(\frac{\sigma}{r})>0$ if and only if $\frac{1}{2} \leq \frac{\sigma}{r}\leq 1,$ then $\frac{1}{\sigma}\leq\frac{2}{r}$ when $-\phi'(\frac{\sigma}{r})>0.$
    By monotonicity of $I^v(x,\cdot)$ at sufficiently small radius and change of variables, 
    \begin{equation}\label{eq:12}
    I_{\phi,v}(x,r)\leq -\int_0^r \frac{2}{r}\phi'(\frac{\sigma}{r})I^v(x,r) \, d\sigma = -2I^v(x,r) \int_0^r\frac{1}{r} \phi'(\frac{\sigma}{r}) \, d\sigma=2I^v(x,r).
    \end{equation}
    Similarly, for $2r\leq R_0$ we have
    \begin{equation*}
        I_{\phi,v}(x,2r)=-\int_0^{2r}\frac{1}{\sigma}\phi'\paren{\frac{\sigma}{2r}}I^v(x,\sigma) \, d\sigma=2\int_{r}^{2r}\frac{I^v(x,\sigma)}{\sigma}\, d\sigma \geq 2\ln 2 I^v(x,r).
    \end{equation*}
    For $i$ sufficiently large such that $2s_i <R_0,$ we have uniform bounds 
    \begin{equation}\label{eq:13}
        2\geq \frac{I_{\phi,v}(x,s_i)}{I^v(x,s_i)}\geq C'\frac{I_{\phi,v}(x,2s_i)}{I^v(x,s_i)}\geq C'',
    \end{equation}
    where the second inequality in \eqref{eq:13} comes from \eqref{eqn:6} and the uniform upper bound on $\Ord_{\phi,v}(x,s_i)$. Thus, by choosing a subsequence, from \eqref{eq:13} we may assume that $\frac{I_{\phi,v}(x,s_i)}{I^v(x,s_i)}\to \mu$ as $i\to\infty$ for some $\mu>0$. By \cite[Remark 55]{dm}, a subsequence of the maps $v_i:B_1(0)\to (Y_C,(s_i^{1-n}I^v(x,s_i))^{-1/2}d)$ given by $v_i(y)=v(x+s_i y)$ converge uniformly on $B_{1/2}(0)$ to a tangent map $v_*:B_{1/2}(0)\to Y_C$. But since $\frac{I_{\phi,v}(x,s_i)}{I^v(x,s_i)}\to \mu$, we must have that $\tilde{v}_i\to \mu^{-1/2} v_*$ uniformly as $i\to \infty$ on $B_{1/2}(0)$, and thus also that $\tilde{h}_i\to \mu^{-1/2}v_*$ as $i\to \infty$. But since each $\tilde{h}_i$ is $(k+1)$-homogeneous, by choosing a subsequence as in Lemma \ref{bd5.8}, we see that $v_*$ must be $(k+1)$-homogeneous, contradicting the assumption $x\in \calS_0^k(v)$.
\end{proof}
\begin{lemma}\label{lemma:d5.3}
    Assume $\Ord_{\phi,v}(0,R_0)\leq \Lambda$ and $\sup_{B_{\sigma_j}(0)}|\nabla u| \leq M$. For any $\theta\in (0,1)$, there exists $C>0$ depending only on $\theta$, $M$, $X$, and $\Lambda$ such that for all $x\in B_{R_0/64}(0)$ and all $0<R\leq R_0/16$, we have
    \begin{equation}\label{eq:zetabound}
        \int_{B_{R/2}(x)\backslash B_{\theta R}(x)}\abs{|y-x|\pder[v]{r}-\BDOrd_{\phi,v}(x,r)v(y)}^2dy\leq CRI_{\phi,v}(x,R)(W_{\theta R/2}^R(x)+R^2).
    \end{equation}
\end{lemma}
\begin{proof}
If $\Ord_{\phi,v}(0, R_0)\leq\Lambda$, then we have $\BDOrd_{\phi,v}(x, r)\leq\Ord_{\phi,v}(x,r)\leq  \Ord_{\phi,v}(x,R_0/16)\leq C(\Lambda+1):=\Lambda'$ for all $r\leq R_0/16$ and all $x\in B_{R_0/64}(0)$. Fix $x\in B_{R_0/64}(0)$ and denote $\zeta(y)=\abs{|y-x|\pder[v]{r}-\BDOrd_{\phi,v}(x, |x-y|)v(y)}^2$ and $\eta(y,r)=\abs{|y-x|\pder[v]{r}-\BDOrd_{\phi,v}(x,  r)v(y)}^2$.  Thus, for all $\theta\in (0,1)$, $r\in(\theta R,R)$ and $0<R\leq R_0/16$ and all $y\in B_r(x)\backslash B_{r/2}(x)$, we have
    \begin{align*}
        \zeta(y)&\leq 2\eta(y,r)+2d^2(v(y),P_0)|\BDOrd_{\phi,v}(x,r)-\BDOrd_{\phi,v}(x,|y-x|)|^2\\
        &= 2\eta(y,r)+2d^2(v(y),P_0)|e^{-Ar}(W_{|y-x|}^r(x)+e^{A|y-x|}\BDOrd_{\phi,v}(x,|y-x|)\\
        &-A(r^2-|y-x|^2))-\BDOrd_{\phi,v}(x,|y-x|)|^2\\
        &\leq 2\eta(y,r)+6d^2(v(y),P_0)((W_{\theta R/2}^R(x))^2+(1-e^{-Ar})^2\BDOrd_{\phi,v}(x,|y-x|)^2+A^2r^4)\\
        &\leq 2\eta(y,r)+6d^2(v(y),P_0)(\Lambda'W_{\theta R/2}^R(x)+(A^2(\Lambda')^2+A^2(R_0/16)^2)r^2)
    \end{align*}
    where in the final inequality we used the inequality $1-e^{Ar}\leq Ar$. Putting this into \eqref{eq:2} and recalling the definition of $I_{\phi,v}(x,r)$, we get
    \begin{align*}
        \der{r}\log \BDOrd_{\phi,v}(x,r)&\geq-A-\frac{2Ar}{\BDOrd_{\phi,v}(x,r)}+\frac{1}{r^2E_{\phi,v}(x, r)}\int_{B_r(x)}-|y-x|^{-1}\phi'\paren{\frac{|y-x|}{r}}\zeta(y)dy\\
        &-6(\Lambda'W_{\theta R/2}^R(x)+(A^2(\Lambda')^2+C_1'^2(R_0/16)^2)r^2)\frac{I_{\phi,v}(x, r)}{r^2E_{\phi,v}(x,r)}\\
        &=-A-\frac{6\Lambda'W_{\theta R/2}^R(x)}{r\BDOrd_{\phi,v}(x, r)}-\frac{6r(2A+A^2(\Lambda')^2+A^2(R_0/16)^2)}{\BDOrd_{\phi,v}(x, r)}\\
        &+\frac{1}{r^2E_{\phi,v}(x, r)}\int_{B_r(x)}-|y-x|^{-1}\phi'\paren{\frac{|y-x|}{r}}\zeta(y)dy
    \end{align*}
    Adding $A$ to both sides and multiplying both sides by $e^{Ar}\BDOrd_{\phi,v}(x, r)$, we obtain
    \begin{align*}
        \der{r}e^{Ar}\BDOrd_{\phi,v}(x,r)&\geq-\frac{6\Lambda'e^{Ar}W_{\theta R/2}^R(x)}{r}-re^{Ar}(2C_1'+(C_1'+C_3)^2(\Lambda')^2+C_1'^2(R_0/16)^2)\\
        &+\frac{1}{rI_{\phi,v}(x,r)}\int_{B_r(x)}|y-x|^{-1}\phi'\paren{\frac{|y-x|}{r}}\zeta(y)dy\\
        &\geq -\frac{6\Lambda'e^{AR_0/16}}{\theta R_0/16}W_{\theta R/2}^R(x)-\der{r}Ar^2e^{Ar}-re^{AR_0}(A^2(\Lambda')^2
        \\&+A^2(R_0/16)^2)+\frac{1}{RI_{\phi,v}(x,R)}\int_{B_r(x)}-|y-x|^{-1}\phi'\paren{\frac{|y-x|}{r}}\zeta(y)dy,
    \end{align*}
    where in the final inequality we have used Lemma \ref{monotoneorder} to get $I_{\phi,v}(x,r)\leq I_{\phi,v}(x,R)$ for all $r\in (\theta R,R)$. Adding $\der{r}Ar^2e^{Ar}$ to both sides and integrating over $r\in (\theta R,R)$, we obtain
    \begin{align*}
    W_{\theta R}^R(x)&\geq -\frac{6\Lambda'e^{AR_0/16}(1-\theta)}{\theta}W_{\theta R/2}^R(x)-e^{AR_0/16}(A^2(\Lambda')^2+A^2(R_0/16)^2)(1/2-\theta^2/2)R^2\\
    &+\frac{1}{RI_{\phi,v}(x,R)}\int_{\theta R}^R\int_{B_r(x)}-|y-x|^{-1}\phi'\paren{\frac{|y-x|}{r}}\zeta(y)dydr\\
    &=-\frac{6\Lambda'e^{AR_0/16}(1-\theta)}{\theta}W_{\theta R/2}^R-e^{AR_0/16}(A^2(\Lambda')^2+A^2(R_0/16)^2)(1/2-\theta^2/2)R^2\\
    &+\frac{1}{RI_{\phi,v}(x,R)}\int_{B_R(x)}\zeta(y)\paren{\int_{\theta R}^R-|y-x|^{-1}\phi'\paren{\frac{|y-x|}{r}}dr}dy
    \end{align*}
    Since $\int_{\theta R}^R-|y-x|^{-1}\phi'\paren{\frac{|y-x|}{r}}dr\geq 2\chi_{B_{R/2}(x)\backslash B_{\theta R}(x)}(y)$, rearranging terms and recalling that $W_{\theta R}^R(x)\leq W_{\theta R/2}^R(x)$ gives us that there exists $C>0$ such that
    \begin{equation*}
        \int_{B_{R/2}(x)\backslash B_{\theta R}(x)}\zeta(y)dy\leq CRI_{\phi,v}(x,R)(W_{\theta R/2}^R(x)+R^2).
    \end{equation*}
\end{proof}
We next study domain variations of the singular component and the variation of the smoothed frequency with respect to its center. The following notation will be used throughout Lemma~\ref{domainvariation} and Lemma~\ref{comparefreq}. For any fixed constant vector $w\in \RR^n$ in the domain coordinates and $\xi\in C_c^{\infty}(\Omega)$ with $\xi\geq 0$ and all $t$ sufficiently small, define the domain variation $F_t:\Omega\to \Omega$ by $F_t(x)=x+t\xi(x)w$ and the corresponding variation $v_t:\Omega\to Y_C$ by $v_t=v\circ F_t$.
\begin{lemma}\label{domainvariation}
    There exists $C>0$ depending only on the constant in the estimates \cite[(29)--(33)]{dm} and the Lipschitz constant of $u$ such that
    \begin{equation}\label{eq:wvar}
        \lim_{t\to 0}\frac{E(v_t)-E(v)}{t}\leq C|w|\int_{\Omega}\xi(d^2(v,P_0)+|\nabla v|^2).
    \end{equation}
\end{lemma}
\begin{proof}
    The proof follows the proof of \cite[Lemma 53]{dm}, except we may use the fact that $\calS(u)$ has measure zero, and we use the following identities for $x\in \calR(u)$:
    \begin{equation*}
        \evalat{\der{t}}{t=0}v^i_t=\evalat{\pder[v_t^i]{y^{\alpha}}\pder[y^{\alpha}]{t}}{t=0}=\pder[v_t^i]{x^{\alpha}}\xi w^{\alpha},
    \end{equation*}
    \begin{equation*}
        \pder[v_t^i]{x^{\beta}}=\pder[v_t^i]{y^{\gamma}}\pder[y^{\gamma}]{x^{\beta}}=\pder[v_t^i]{y^{\gamma}}(\delta_{\beta}^{\gamma}+t\pder[\xi]{x^{\beta}}w^{\gamma})
    \end{equation*}
    \begin{equation*}
        \evalat{\der{t}}{t=0}\pder[v_t^i]{x^{\beta}}=\smpder[v^i]{x^{\delta}}{x^{\beta}}\xi w^{\delta}+\pder[v^i]{x^{\gamma}}\pder[\xi]{x^{\beta}}w^{\gamma}
    \end{equation*}
    where $y(x)=F_t(x)$.
\end{proof}
\begin{lemma}\label{comparefreq}
    There exist $(C,R_1)(\Lambda,n,M,\sigma_j,X)>0$ with the following property: If $r\leq R_1$ and $x_1,x_2\in B_{r/8}(0)$, then for all $y,z$ on the line segment between $x_1$ and $x_2$, we have
    \begin{equation}\label{eq:pinchingbound}
    |\BDOrd_{\phi,v}(y,r)-\BDOrd_{\phi,v}(z,r)|\leq C(W(x_1)+W(x_2)+r)^{1/2}\frac{|y-z|}{|x_1-x_2|}.
\end{equation}
\end{lemma}
\begin{proof}
By computations similar to those in the proof of \cite[Lemma 3.2]{dees}, if we take $\xi=\phi(|x|/r)$, we have
\begin{eqnarray*}
    \evalat{\der{t}}{t=0}E(v_t)&=&\frac{1}{r}\int_{\Omega}2\xi'(|x|)\abrac{\partial_{\nu_x}v,\partial_wv}-|\nabla v|^2\xi'(|x|)\frac{x}{|x|}\cdot w \, dx\\
    &=&\partial_wE_{\phi,v}(0,r)+\frac{2}{r}\int_{\Omega}\phi'\paren{\frac{|x|}{r}}\abrac{\partial_{\nu_x}v,\partial_wv}.
\end{eqnarray*}
As in \cite{dees}, we have
\begin{equation*}
    \partial_wI_{\phi,v}(x,r)=-2\int_{\Omega}|y-x|^{-1}\phi'\paren{\frac{|x-y|}{r}}\abrac{v(y),\partial_w v(y)}dy
\end{equation*}
By \eqref{eq:wvar}, we have for $r\leq R_0$:
\begin{align*}
    \partial_w \BDOrd_{\phi,v}(x,r) =&\frac{r}{I_{\phi,v}(x,r)}\partial_{w}E_{\phi,v}(x,r)-
    \frac{\BDOrd_{\phi,v}(x,r)}{I_{\phi,v}(x,r)}\partial_{w}I_{\phi,v}(x,r)\\
    \leq&\frac{2}{I_{\phi,v}(x,r)}\int_{\Omega}\abrac{\partial_{\eta_x}v,\partial_w v}\, d\mu_x+\frac{Cr|w|}{I_{\phi,v}(x,r)} \int_{\Omega} \phi(\frac{|y-x|}{r})\bigl(d^2(v,P_0)+|\nabla v|^2\bigr)\, dy\\
    &-\frac{2\BDOrd_{\phi,v}(x,r)}{I_{\phi,v}(x,r)}\int_{\Omega}\abrac{v,\partial_w v}\, d\mu_x\\
    \leq&\frac{2}{I_{\phi,v}(x,r)}\paren{\int_{\Omega}\abrac{\partial_{\eta_x}v,\partial_{w}v}d\mu_x-\BDOrd_{\phi,v}(x,r)\int_{\Omega}\abrac{v,\partial_wv}d\mu_x}+2C|w|\BDOrd_{\phi,v}(x,r)\\
    &+\frac{2Cr|w|}{I_{\phi,v}(x,r)}\int_{\Omega}d^2(v(y),P_0)\phi\paren{\frac{|x-y|}{r}}dy\\
    \leq& \underbrace{\frac{2}{I_{\phi,v}(x,r)}\int_{\Omega}\abrac{\partial_{\eta_x}v(y)-\BDOrd_{\phi,v}(x,r)v(y),\partial_wv(y)}d\mu_x}_{:=\calM}+2C|w|\BDOrd_{\phi,v}(x,r)+2Cr^2|w|,
\end{align*}
where $d\mu_x(y):=-|y-x|^{-1}\phi'(\frac{|y-x|}{r})\, dy.$ To derive the last inequality, we bound $\int_0^r I^v(x,\sigma)\phi(\frac{\sigma}{r})\, d\sigma,$ which is equivalent to $\int_{\Omega}d^2(v(y),P_0)\phi(\frac{|y-x|}{r})\, dy.$ By the monotonicity of $I^v(x,\sigma),$ 
\[
I_{\phi,v}(x,r) = 2\int_{r/2}^r \frac{I^v(x,\sigma)}{\sigma}\, d\sigma \geq 2I^v(x,r/2)\int_{r/2}^r\frac{d\sigma}{\sigma}=2\ln2 I^v(x,r/2) \text{ and } \] \[\int_0^{r/2} I^v(x,\sigma) \, d\sigma \leq \frac{r}{2}I^v(x,r/2),
\]
which follows that $$\int_0^{r/2}I^v(x,\sigma)\, d\sigma \leq \frac{r}{4\ln2}I_{\phi,v}(x,r).$$
Combining with $\int_{r/2}^r I^v(x,\sigma)\phi(\frac{\sigma}{r}) \, d\sigma \leq\frac{r}{2}I_{\phi,v}(x,r),$ we finish bounding.

Now we take $w=x_1-x_2$ for $x_1,x_2\in B_{r/8}(0)$ with $|w|=|x_1-x_2|$. For now we will assume $4r\leq R_0/16$. We aim to bound $\calM$ using similar computations as in the proof of \cite[Theorem 5.3]{dees}. The changes we make are as follows. We estimate $\calE_4(z):=\BDOrd_{\phi,v}(x_1,|z-x_1|)-\BDOrd_{\phi,v}(x_1,r)$ and $\calE_5(z):= \BDOrd_{\phi,v}(x_2,|z-x_2|)-\BDOrd_{\phi,v}(x_2,r)$ for all $z\in\supp(\mu_x)$, for all $x$ on the line segment between $x_1$ and $x_2$,
\begin{align*}
    |\calE_4(x)|&=|\BDOrd_{\phi,v}(x_1,r)-\BDOrd_{\phi,v}(x_1,|z-x_1|)|=|e^{-Ar}(W_{|z-x_1|}^r(x_1)+e^{A|z-x_1|}(\BDOrd_{\phi,v}(x_1,|z-x_1|)\\
    &+A|z-x_1|^2))-Ar^2-\BDOrd_{\phi,v}(x_1,|z-x_1|)|\\
    &\leq W_{|z-x_1|}^r(x_1)+(1-e^{-Ar})\BDOrd_{\phi,v}(x_1,|z-x_1|)+Ar^2\leq W(x_1)+Cr,
\end{align*}
where $W(x):=W^{4r}_{r/8}(x)$ and $C$ denotes a generic constant depending only on $(n,\sigma_j,M,Y_C,\Lambda)$ and similarly
\begin{equation*}
    |\calE_5(x)|\leq W(x_2)+Cr.
\end{equation*}
By \eqref{eq:1} and \eqref{eq:9.5}, we have for $\calE:=\BDOrd_{\phi,v}(x_1,r)-\BDOrd_{\phi,v}(x_2,r),$
\begin{align*}
    \abs{\int\calE
    \abrac{v,\partial_{\eta_x}v}d\mu_x-\BDOrd_{\phi,v}(x,r)\int\calE d^2(v,0_x)d\mu_x}&=\abs{\calE}\abs{\int-\phi'\paren{\frac{|x-y|}{r}}\abrac{\partial_{\nu_x}v,v}-rE_{\phi,v}(x,r)}\\
    &\leq C|\calE|r^2I_{\phi,v}(x,r).
\end{align*}
In \cite{dees}, the LHS is simply 0. This all gives for $\calE_3(z):=\calE+\calE_4(z)-\calE_5(z),$
\begin{align*}
    \left|(C)\right|&:=\left|2I_{\phi,v}(x,r)^{-1}\left[\int \abrac{\calE_3 v, \partial_{\eta_x}v} \, d\mu_x -\BDOrd_{\phi,v}(x,r) \int \calE_3 |v|^2 \, d\mu_x \right]\right|\\
    &\leq Cr^2|\calE|+(W(x_
    1)+W(x_2)+Cr)2I_{\phi,v}(x,r)^{-1}\paren{\int|v||\nabla v|d\mu_x+\BDOrd_{\phi,v}(x,r)\int |v|^2d\mu_x}\\
    &=Cr^2|\calE|+(W(x_1)+W(x_2)+Cr)2I_{\phi,v}(x,r)^{-1}\paren{\int|v||\nabla v|d\mu_x+rE_{\phi,v}(x,r)}\\
    &\leq Cr^2|\calE|+(W(x_1)+W(x_2)+Cr)2I_{\phi,v}(x,r)^{-1}\paren{\int|v|^2/2+|\nabla v|^2/2d\mu_x+rE_{\phi,v}(x,r)}\\
    &\leq Cr^2|\calE|+(W(x_1)+W(x_2)+Cr)2I_{\phi,v}(x,r)^{-1}(1/2I_{\phi,v}(x,r)+(1/2+r)E_{\phi,v}(x,2r))\\
    &\leq Cr^2|\calE|+(W(x_1)+W(x_2)+Cr)(1+CI_{\phi,v}(x,r)^{-1}E_{\phi,v}(x,2r)).
\end{align*}
By Lemma \ref{ineq:ord} and Equation \eqref{eqn:6}, we have
\begin{equation*}
    rI_{\phi,v}(x,r)^{-1}E_{\phi,v}(x,2r)\leq CI_{\phi,v}(x,r)^{-1}I_{\phi,v}(x,2r)\leq C.
\end{equation*}
Thus, we have
\begin{equation*}
    \left|(C)\right|\leq C(r^2|\calE|+W(x_1)+W(x_2)+r).
\end{equation*}
Define $\calE_{\ell}(z):=\partial_{\eta_{x_{\ell}}}v-\BDOrd_{\phi,v}(x_{\ell},|z-x_{\ell}|)v(z)$ for $\ell=1,2.$ We derive the same bounds for $|(A)|=\left|2I_{\phi,v}(x,r)^{-1}\int \abrac{(\calE_1-\calE_2),\partial_{\eta_x}v} \, d\mu_x \right|$ and $|(B)|=\left|\frac{2rE_{\phi,v}(x,r)}{I_{\phi,v}(x,r)^2}\int \abrac{(\calE_1-\calE_2),v} \, d\mu_x \right|$ as in the proof of \cite[Theorem 5.2]{dees} by the same methods. Result (\ref{eq:zetabound}) by letting $R=4r$ and $\theta=\frac{1}{16}$ gives us
\begin{equation*}
    \int |\calE_{\ell}|^2d\mu_x\leq \frac{4}{r}\int_{B_{2r}(x_{\ell})\backslash B_{r/4}(x_{\ell})}|\calE_{\ell}|^2dy\leq C(W(x_{\ell})+r^2)\leq C(W(x_{\ell})+r).
\end{equation*}
Thus, we have
\begin{equation*}
    |\partial_w\BDOrd_{\phi,v}(x,r)|\leq C(r^2|\calE|+r+W(x_1)+W(x_2)+(W(x_1)+W(x_2)+r)^{1/2}).
\end{equation*}
By Lemma \ref{ineq:ord}, we have $(W(x_1)+W(x_2)+r)^{1/2}\leq C$, so we finally have
\begin{equation}\label{eq:6}
    |\partial_w\BDOrd_{\phi,v}(x,r)|\leq C(r^2|\calE|+(W(x_1)+W(x_2)+r)^{1/2}).
\end{equation}
Integrating \eqref{eq:6} along the line segment between $x_1$ and $x_2$, we have
\begin{equation*}
    |\calE|\leq C(r^2|\calE|+r+(W(x_1)+W(x_2)+r)^{1/2}).
\end{equation*}
For $r\leq \min\{R_0/64,1/\sqrt{2C}\}:=R_1$, we may subtract $Cr^2|\calE|$ from both sides and multiply both sides by 2 to get
\begin{equation*}
    |\calE|\leq C(r+(W(x_1)+W(x_2)+r)^{1/2}).
\end{equation*}
Putting this into \eqref{eq:6}, we get
\begin{equation*}
    |\partial_w\BDOrd_{\phi,v}(x,r)|\leq C(W(x_1)+W(x_2)+r)^{1/2}.
\end{equation*}
Integrating along the line segment from $y$ to $z$, we get
\begin{equation*}
    |\BDOrd_{\phi,v}(y,r)-\BDOrd_{\phi,v}(z,r)|\leq C(W(x_1)+W(x_2)+r)^{1/2}\frac{|y-z|}{|x_1-x_2|}.
\end{equation*}
\end{proof}
\begin{lemma}\label{compactness}
    Suppose that $\{u_i\}_{i=1}^{\infty}=\{(V_i,v_i):B_{\sigma_j}(0)\to \RR^j\times Y_C^{N-j}\}_{i=1}^{\infty}$ are harmonic maps with $\sup_{B_{\sigma_j}(0)}|\nabla u_i| \leq M$ and $\Ord^{v_i}(0,R_0)\leq \Lambda$. Suppose $\{\sigma_i\}_{i=1}^{\infty}$ are such that $0<\sigma_i<R_0$ and $\sigma_i\to 0$ as $i\to \infty$ and $v_i(x_i)=O_Y$ for $x_i\in B_{\sigma_i/64}(0)$. Let $\tilde{v}_i:B_1(0)\to (Y_C,d_i=\lambda_i^{-1}d_{Y_C})$ be such that $\tilde{v}_i(x)=v_i(\sigma_ix)$, where $\lambda_i=(\sigma_i^{1-n}I_{\phi,v_i}(0,\sigma_i))^{1/2}$. Then there exists a nonconstant harmonic map $v_{\infty}:{B_{1/16}(0)}\to (Y_{\infty},d_{\infty})=(Y_C,d_{Y_C})$ such that $\tilde{v}_i|_{B_{1/16}(0)}$ converges to $v_{\infty}$ uniformly in the pullback sense and the (directional) energy densities of the $\tilde{v}_i$'s converge weakly to those of $v_{\infty}$ on $B_{1/64}(0)$.
\end{lemma}
\begin{proof}
Since $Y_C$ is conical, for each $i$ the dilation $\delta_i:(Y_C,d_i) \to (Y_C,d_{Y_C})$ by $\delta_i(p)=\lambda_i^{-1}p$ is an isometry. Define $\hat{v}_i:=\delta_i \circ \tilde{v}_i$ by $\hat{v}_i(x)=\lambda_i^{-1}v_i(\sigma_ix).$ We firstly prove compactness for the maps $\hat{v}_i$ in the fixed target $(Y_C,d_{Y_C}).$ After passing to the subsequence, the resulting uniform convergence is equivalent to the pullback convergence of the original maps $\tilde{v}_i:(B_1(0),g_{\sigma_i}) \to (Y_C,d_i=\lambda_i^{-1}d_{Y_C})$ since $\tilde{v}_i$ and $\hat{v}_i$ have the same energy and directional energy density measures via the isometry.

    By Equation \eqref{eqn:6}, Lemma \ref{monotoneheight}, and the fact that $(I^{v_i})'(x,\sigma)>0$ for all $\sigma<R_0$, we have \[I_{\phi,v_i}(0,\sigma_i)\leq CI_{\phi,v_i}(0,\sigma_i/64)\leq CI_{\phi,v_i}(x_i,\sigma_i/16)\leq CI^{v_i}(x_i,\sigma_i/8).\] On the other hand, we have \[I^{v_i}(x_i,\sigma_i/8)\leq CI_{\phi,v_i}(x_i,\sigma_i/4)\leq CI_{\phi,v_i}(0,\sigma_i).\] Thus, there is a constant $C$ independent of $i$ such that $C^{-1}\leq I^{\hat{v}_i}(\sigma_i^{-1}x_i,1/8)\leq C$. Also, we have 
    \[E^{\hat{v}_i}(\sigma_i^{-1}x_i,1/8)=\int_{B_{1/8}(\sigma_i^{-1}x_i)}|\nabla \hat{v}_i|^2 \leq \int_{B_1(0)}\phi(|y|)|\nabla \hat{v}_i|^2 = E_{\phi,\hat{v}_i}(0,1)=\BDOrd_{\phi,v_i}(0,\sigma_i)\leq \Lambda.\] 
    Let $w_i:B_{1/8}(\sigma^{-1}_ix_i)\to Y_C$ be the harmonic map with $w_i|_{\partial B_{1/8}(\sigma_i^{-1}x_i)}=\hat{v}_i|_{\partial B_{1/8}(\sigma_i^{-1}x_i)}$. Then, we have \[E^{w_i}(\sigma_i^{-1}x_i,1/16)\leq E^{w_i}(\sigma_i^{-1}x_i,1/8)\leq E^{\hat{v}_i}(\sigma^{-1}x_i,1/8)\leq \Lambda\] and $I^{w_i}(\sigma_i^{-1}x_i,1/8)=I^{\hat{v}_i}(\sigma_i^{-1}x_i,1/8)$. By subharmonicity of $d^2(w_i,O_Y)$, there exists $C$ such that \[\sup_{B_{1/16}(0)}d^2(w_i,O_Y)\leq CI^{w_i}(\sigma_i^{-1}x_i,1/8)=CI^{\hat{v}_i}(\sigma_i^{-1}x_i,1/8)\leq C.\] Also, there exists $C$ such that $\sup_{B_{1/16}(0)}|\nabla w_i|^2\leq \sup_{B_{1/8}(\sigma_i^{-1}x_i)}|\nabla w_i|^2\leq C$. Thus, by the Arzel\'{a}-Ascoli compactness theorem, there is an unrelabeled subsequence such that $w_i|_{B_{1/16}(0)}$ converges uniformly to a harmonic map $v_{\infty}:B_{1/16}(0)\to Y_C$. By \cite[Proposition 48]{dm}, there exists $C$ independent of $i$ such that \[\sup_{B_{1/16}(0)}d^2(\hat{v}_i(x),w_i(x))\leq\sup_{B_{5/64}(0)}d^2(\hat{v}_i(x),w_i(x))\leq  C\sigma_i^2I^{\hat{v}_i}(\sigma_i^{-1}x_i,1/8)\leq C\sigma_i^2,\] so $\hat{v}_i\to v_{\infty}$ uniformly on $B_{1/16}(0)$. By $\eqref{eqn:6}$, we have $I_{\phi,\hat{v}_i}(0,1/16)\geq C$ for some $C$ independent of $i$, so by uniform convergence of the $\hat{v}_i's$ to $v_{\infty}$, we also have $I_{\phi,v_{\infty}}(0,1/16)\geq C$, so $v_{\infty}$ is nonconstant. By \cite[Theorem 3.11]{ks2}, the (directional) energy densities of the $w_i$'s converge weakly on $B_{1/16}(0)$ to those of $v_{\infty}$. By \cite[Lemma 3.8]{ks2}, for any smooth vector field $Z$ on $B_{1/16}(0)$ and any $f\in C_c^{\infty}(B_{1/16}(0))$, we have \[\int_{B_{1/16}(0)}f|(v_{\infty})_*(Z)|^2d\mu\leq\liminf_{i\to\infty}\int_{B_{1/16}(0)}f|(\hat{v}_i)_*(Z)|^2d\mu\] and \[\int_{B_{1/16}(0)}f|\nabla v_{\infty}|^2d\mu\leq \liminf_{i\to\infty}\int_{B_{1/16}(0)}f|\nabla \hat{v}_i|^2d\mu.\] 

    It remains to be shown that there is a subsequence $\{\hat{v}_{i_j}\}_{j\in \NN}$ such that the (directional) energy densities of the $\hat{v}_{i_j}$'s converge to those of $v_{\infty}$ on $B_{1/64}(0)$. Set $i_{j,1}=j$ for all $j\in \NN$. Let $2\leq m\in \NN$ be given, and suppose we have defined the sequence $\{i_{j,m-1}\}_{j\in \NN}$. For each $i\in \NN$ there is some $r_i\in (1/32(1+1/m),1/32(1+2/m))$ such that $\int_{\partial B_{r_i\sigma_i}(x_i)}|\nabla v_i|^2d\Sigma\leq \frac{32m}{\sigma_i}E^{v_i}(x_i,\sigma_i/8)$. Following the proof of Inequality \cite[(132)]{dm}, there exists $C>0$ independent of $i$ such that $|E^{\hat{v}_i}(\sigma_i^{-1}x_i,r_i)-E^{w_i}(\sigma_i^{-1}x_i,r_i)|\leq C\sigma_i$. Then, there exists a subsequence $\{i_{j,m}\}_{j\in \NN}$ of $\{i_{j,m-1}\}_{j\in \NN}$ such that $r_{i_{j,m}}\to r_{0,m}\in [1/32(1+1/m),1/32(1+2/m)]$ as $j\to\infty$. Then, we have 
    \begin{align*}
        E^{v_{\infty}}(x_0,1/32(1+2/m))\geq& E^{v_{\infty}}(x_0,r_{0,m})=\lim_{j\to\infty}E^{w_{i_{j,m}}}(x_0,r_{0,m}).
    \end{align*}
    Since there exists $L$ such that $\sup_{B_{1/16}(\sigma_i^{-1}x_i)}|\nabla w_i|\leq L$ for all $i\in \NN$, we have \[\lim_{j\to\infty}E^{w_{i_{j,m}}}(x_0,r_{0,m})-\lim_{j \to \infty} E^{w_{i_{j,m}}}(\sigma^{-1}_ix_i,r_{i_{j,m}})=0.\] Therefore, we have
    \begin{align*}
        E^{v_{\infty}}(x_0,1/32(1+2/m))&\geq \lim_{j\to\infty}E^{w_{i_{j,m}}}(\sigma^{-1}_{i_{j,m}}x_{i_{j,m}},r_{i_{j,m}})=\lim_{j\to\infty}E^{\hat{v}_{i_{j,m}}}(\sigma^{-1}_{i_{j,m}}x_{i_{j,m}},r_{i_{j,m}})\\
        &\geq \limsup_{j\to \infty}E^{\hat{v}_{i_{j,m}}}(x_0,1/32).
    \end{align*}
    For all $j\in \NN$, let $i_j=i_{j,j}$. Then, for all $m\in \NN$, $\{i_j\}_{j\geq m}$ is a subsequence of $\{i_{j,m}\}_{j\in \NN}$. Thus, for all $m\in \NN$ we have $E^{v_{\infty}}(x_0,1/32(1+2/m))\geq \limsup_{j\to\infty}E^{\hat{v}_{i_j}}(x_0,1/32)$. Sending $m\to\infty$, we have $\limsup_{j\to\infty}E^{\hat{v}_{i_j}}(x_0,1/32)\leq E^{v_{\infty}}(x_0,1/32)$. Combining this with lower semicontinuity of the energy, we have that $E^{v_{\infty}}(x_0,1/32)=\lim_{j\to\infty}E^{\hat{v}_{i_j}}(x_0,1/32)$, which implies the result since $B_{1/32}(x_0)\supset B_{1/64}(0)$.
\end{proof}

\section{$k$-homogeneity and pinching}\label{sec:k-hom}
In this section, we use following lemmas to show that if the frequency pinching is small at sufficiently many effectively spanning points, the singular set lies close to an affine subspace. The arguments use the sequence of blow-up maps of the singular component and reduce its pullback limit as a harmonic map into a conical $F$-connected complex in order to apply the results in \cite{dees} and \cite{bd}. We begin with the concept of ``effectively spanning a $k$-dimensional affine subspace'', which is important in our results since it is stable under blow-up limits.
\begin{definition}
   Let $\{x_0,...,x_k\} \in B_1(0) \subset \RR^n.$ We say that this set {\bf $\rho$-effectively spans} a $k$-dimensional affine subspace if for all $i=1,...,k,$
   \[
   x_i \notin B_{\rho}(x_0+\text{span}(x_1-x_0,...,x_{i-1}-x_0)).
   \]
   Given a set $F \subset B_1(0),$ we say that {\bf $F \ \rho$-effectively spans} a $k$-dimensional affine subspace if there exists a set $\{x_0,...,x_k\} \subset F$ that $\rho$-effectively spans a $k$-dimensional affine subspace.
\end{definition}

\begin{lemma}\label{bd5.8}
    Let $u$ be a harmonic map to an $N$-dimensional conical DM-complex $Y_C$ with splitting $u=(V,v):B_{\sigma_j}(0) \to (\mathbb{R}^j \times Y_C^{N-j},d_G)$ satisfying $\Ord_{\phi,v}(0,R_0)\leq \Lambda <\infty$ and $\sup_{B_{\sigma_j}(0)}|\nabla u| \leq M,$ where $R_0>0$ is as in Lemma \ref{monotoneorder}. Let $\rho \in (0, \frac{1}{4}]$ and $\eta>0,$ there exist $\sigma_{\ref{bd5.8}}, \delta_{\ref{bd5.8}}=(\sigma_{\ref{bd5.8}}, \delta_{\ref{bd5.8}})(n,\rho,X,\Lambda,M)>0$ such that for all $0<\sigma<\sigma_{\ref{bd5.8}}$ and $0<\delta<\delta_{\ref{bd5.8}}$, if $F=\{x \in B_{\sigma}(0): W^{2\sigma}_{\rho\sigma}(x,v)
    < \delta \} \ \sigma\rho$-effectively spans a $k-$dimensional subspace $V,$ then
    \[
    B_{2\sigma}(0) \cap \mathcal{S}^k_{0,\eta,\delta\sigma}(v) \subset B_{2\sigma\rho}(V).
    \]
\end{lemma}
\begin{proof}
    Suppose no such $\sigma,\delta>0$ exist. Then there exists a sequence of harmonic maps $\{u_i=(V_i,v_i): B_{\sigma_j}(0) \to \mathbb{R}^j\times Y_C^{N-j}\},$ where $Y_C^{N-j}$ is a conical $F$-connected complex (cf. \cite[Remark 23]{dm}) and $v_i$ is the non-harmonic singular component, such that $\Ord_{\phi,v_i}(0,R_0)\leq \Lambda$ and $\sup_{B_{\sigma_j}(0)}|\nabla u_i|\leq M$, and for $\sigma_i<\min \{\frac{1}{i}, R_0/192\}$,
    \begin{itemize}
        \item [(i)] there exists $\{x_0^i,...,x_k^i\} \subset B_{\sigma_i}(0) \ \rho\sigma_i$ -- effectively spanning an affine $k$-plane $V^i$
        \item [(ii)] $\Ord_{\phi,v_i}(x^i_j,2\sigma_i)-\Ord_{\phi,v_i}(x^i_j,\rho\sigma_i) <\frac{1}{i}$
        \item [(iii)] there exists $y^i \in \mathcal{S}^k_{0,\eta,\sigma_i/i}(v_i)\cap B_{2\sigma_i}(0)\setminus B_{2 \rho \sigma_i}(V^i)$ with $v_i(y_i)=O_Y.$
    \end{itemize}
    We have for all $i\in \NN$,
    {\small\begin{align*}
        |\BDOrd_{\phi,v_i}(x_j^i,2\sigma_i)-\BDOrd_{\phi,v_i}(x_j^i,\rho\sigma_i)|&=|e^{-2A\sigma_i}\Ord_{\phi,v_i}(x_j^i,2\sigma_i)-A(2\sigma_i)^2-e^{-A\rho\sigma_i}\Ord_{\phi,v_i}(x_j^i,\rho\sigma_i)+A(\rho\sigma_i)^2|\\
        &=|e^{-2A\sigma_i}W_{\rho\sigma_i}^{2\sigma_i}(x_j^i,v_i)+(e^{-2A)\sigma_i}-e^{A\rho\sigma_i})\Ord_{\phi,v_i}(x_j^i,\rho\sigma_i)-A\sigma_i^2(4-\rho^2)|\\
        &\leq W_{\rho\sigma_i}^{2\sigma_i}(x_j^i,v_i)+A(4-\rho^2)\sigma_i^2+A(2-\rho)\sigma_i\Ord_{\phi,v_i}(x_j^i,\rho\sigma_i).
    \end{align*}}
    By Lemma \ref{monotoneorder} and Lemma \ref{ineq:ord}, we have
    \begin{align*}
        \Ord_{\phi,v_i}(x_j^i,\rho\sigma_i)\leq \Ord_{\phi,v_i}(x_j^i,\sigma_i/4)\leq C_{4.2}(1+\Ord_{\phi,v_i}(0,4\sigma_i))\leq C_{4.2}(1+\Lambda).
    \end{align*}
    Thus, since $\sigma_i\to 0$ and $W_{\rho\sigma_i}^{2\sigma_i}(x_j^i,v_i)\to 0$ as $i\to \infty$, we have that $|\BDOrd_{\phi,v_i}(x_j^i,2\sigma_i)-\BDOrd_{\phi,v_i}(x_j^i,\rho\sigma_i)|\to 0$ as $i\to\infty$. Then the blow-up maps $\tilde{v}_i:B_1(0) \to (Y_C^{N-j},d_i:=\frac{1}{\lambda_i}d_{Y_C})$ where $\lambda_i=((192\sigma_i)^{1-n}I_{\phi,v_i}(0,192\sigma_i))^{\frac{1}{2}}$ and $\tilde{v}_i(x)=v_i(192\sigma_ix)$ satisfy:
    \begin{itemize}
    \item[(i)] $\{(192\sigma_i)^{-1}x_0^i,...,(192\sigma_i)^{-1}x^i_k\} \subset B_{1/192}(0) \ \rho/192$ -- effectively span an affine $k$-space $\sigma_i^{-1}V^i$
    \item[(ii)] $\BDOrd_{\phi, \tilde{v}_i}((192\sigma_i)^{-1}x^i_j,\frac{1}{96})-\BDOrd_{\phi,\tilde{v}_i}((192\sigma_i)^{-1}x^i_j,\frac{\rho}{192}) \to 0$ as $i\to\infty$
    \item[(iii)] there exists $(192\sigma_i)^{-1}y^i \in \mathcal{S}^k_{0,\eta,1/192i}(\tilde{v}_i) \cap B_{1/96}(0) \setminus B_{\rho/96}((192\sigma_i)^{-1}V^i)$ such that $\tilde{v}_i((192\sigma_i)^{-1}y^i)=O_Y.$
    \end{itemize}
    By Lemma \ref{compactness}, an unrelabeled subsequence of the $\tilde{v}_i's$ converges uniformly on $B_{1/64}(0)$ to a harmonic map $v_{\infty}$ such that the energy densities of the $\tilde{v}_i's$ converge weakly to the energy density of $v_{\infty}$. Then, taking a subsequence if necessary, the collection $\{(192\sigma_i)^{-1}x_0^i,...,(192\sigma_i)^{-1}x^i_k\}$ also converges to a collection $\{x_0,...,x_k\} \subset \overline{B_{1/192}(0)}$ which $\rho/192$-effectively spans an affine $k$-space $V^{\infty}.$ Since $B_{1/96}(x_j)\subset B_{1/64}(0)$, and \[|\cdot-(192\sigma_i)^{-1}x_j^i|^{-1}\phi'(|\cdot-(192\sigma_i)^{-1}x_j^i|/r)d^2(\tilde{v}_i,O_Y)\to |\cdot-x_j|^{-1}\phi'(|\cdot-x_j|/r)d_{\infty}^2(v_{\infty},O_Y)\] in $L^{\infty}(B_{1/64}(0))$ as $i\to\infty$, we have that \[\lim_{i\to\infty}I_{\phi,\tilde{v}_i}((192\sigma_i)^{-1}x_j^i,r)=I_{\phi,v_{\infty}}(x_j,r)\] for any $r\leq 1/96$. Also, since $\phi(|\cdot-(192\sigma_i)^{-1}x^i_j|/r)$ converges uniformly to $\phi(|\cdot-x_j|/r)$ as $j\to\infty$, we have $\lim_{i\to\infty}E_{\phi,\tilde{v}_i}((192\sigma_i^{-1})x_j^i,r)=E_{\phi,v_{\infty}}(x_j,r)$ for any $r\leq 1/96$. Therefore, we have $\BDOrd_{\phi,v_{\infty}}(x_j,\frac{1}{96})=\BDOrd_{\phi,v_{\infty}}(x_j,\rho/192)$ for any $j=0,...,k$ and, possibly after taking another subsequence, $(192\sigma_i)^{-1}y^i$ converges to some point $y$ not in $V^{\infty}$ such that $v_{\infty}(y)=O_Y.$ By \cite[Lemma 5.4]{bd}, $v_{\infty}$ is homogeneous of some degree $\alpha_j \geq 1$ on $B_{1/96}(x_j)$ for $j=0,...,k.$ From \cite[Lemma 5.6]{bd} we have that $\alpha_0=...=\alpha_k=\alpha \geq 1$ and $v_{\infty}$ is homogeneous of degree $\alpha$ about $x_j$ for $j=0,...,k.$ Thus, $v_{\infty}$ is $k$-homogeneous about a $k$-plane $V^{\infty}.$ Since $v_{\infty}(y)=O_Y$ and $y \notin V^{\infty},$ from \cite[Lemma 5.7]{bd}, any tangent map $v_*$ of $v_{\infty}$ at $y$ is $(k+1)$-homogeneous. In particular, there is some $r_*$ so that $v_{\infty}$ is $(\frac{\eta}{4},r_*,k+1)$-homogeneous at $y,$ which contradicts the following claim.
\begin{claim}
    For every $(k+1)$-homogeneous map $h,$
    \[
    \sup_{B_r(y)}d_{\infty}(v_{\infty}(x),h(x-y)) \geq \frac{\eta}{2}\left(\frac{I_{\phi,v_{\infty}}(y,r)}{r^{n-1}}\right)^{\frac{1}{2}},
    \]
    where $r:=r_*.$
\end{claim}
\begin{proof}
    Suppose on the contrary that for some $r>0,$ there exists a $(k+1)$-homogeneous map $h$ with
    \begin{equation}\label{5.8:eq1}
    \sup_{B_r(y)}d_{\infty}(v_{\infty}(x),h(x-y)) < \frac{\eta}{2}\left(\frac{I_{\phi,v_{\infty}}(y,r)}{r^{n-1}}\right)^{\frac{1}{2}}
    \end{equation}
    Since $\tilde{v}_i \to v_{\infty}$ uniformly, for each fixed $r>0, \ I_{\phi,\tilde{v}_i}((192\sigma_i)^{-1}y^i,r) \to I_{\phi,v_{\infty}}(y,r),$ so it follows that for $i$ sufficiently large,
    \begin{equation}\label{5.8:eq2}
    \frac{2}{3}\left(\frac{I_{\phi,v_{\infty}}(y,r)}{r^{n-1}}\right)^{\frac{1}{2}} \leq \left(\frac{I_{\phi,\tilde{v}_i}((192\sigma_i)^{-1}y^i,r)}{r^{n-1}}\right)^{\frac{1}{2}} \leq \frac{3}{2}\left(\frac{I_{\phi,v_{\infty}}(y,r)}{r^{n-1}}\right)^{\frac{1}{2}}.
    \end{equation}
    Since $h^{(192\sigma_i)^{-1}y^i}_i:=\delta_i^{-1} \circ h(x-(192\sigma_i)^{-1}y^i)$ converges to $h^{y}$ uniformly in pullback sense as $i \to \infty,$ for sufficiently large $i,$ from (\ref{5.8:eq1}) and (\ref{5.8:eq2}) we have that
    \begin{align*}
        \sup_{B_r((192\sigma_i)^{-1}y^i)} d_i(\tilde{v}_i,h^{(192\sigma_i)^{-1}y^i}_i) 
        &\leq\frac{7\eta}{8}\left(\frac{I_{\phi,\tilde{v}_i}((192\sigma_i)^{-1}y^i,r)}{r^{n-1}}\right)^{\frac{1}{2}},
    \end{align*}
    which contradicts the fact $(192\sigma_i)^{-1}y^i \in \mathcal{S}^k_{0,\eta,1/192i}(\tilde{v}_i)$ when $\frac{1}{192i} <r$.
\end{proof}
\end{proof}

\begin{lemma}\label{bd5.10}
Let $u$ be a harmonic map to an $N$-dimensional conical DM-complex $Y_C$ with splitting $u=(V,v):B_{\sigma_j}(0) \to (\mathbb{R}^j \times Y_C^{N-j},d_G)$ satisfying $\Ord_{\phi,v}(0,R_0) \leq \Lambda, \sup_{B_{\sigma_j}(0)}|\nabla u| \leq M$. Let $\rho \in (0,1] $ and $\epsilon>0,$ then there exist $\sigma_{\ref{bd5.10}},\delta_{\ref{bd5.10}}=(\sigma_{\ref{bd5.10}},\delta_{\ref{bd5.10}})(n,\rho,\epsilon,X,\Lambda,M)>0$ such that if $0<\sigma<\sigma_{\ref{bd5.10}}$ and $\ O_Y \in v(B_{2\sigma}(0))$ and $\sup_{x \in B_{2\sigma}(0)} \Ord_{\phi,v}\paren{x,2\sigma}=\Gamma$ and $F=\{x \in B_\sigma(0):\Ord_{\phi,v}(x,\sigma\rho) >\Gamma-\delta_{\ref{bd5.10}}\}$ 
$\sigma \rho$-effectively spans a $k$-dimensional affine space $V,$ then for all $x \in V \cap B_\sigma(0), \ \Ord_{\phi,v}(x,2\sigma \rho^2)>\Gamma -\epsilon.$
\end{lemma}
\begin{proof}
    Assume that no such $\sigma_{\ref{bd5.10}},\delta_{\ref{bd5.10}}>0$ exist.  Then there exists a sequence of harmonic maps $\{u_i=(V_i,v_i): B_{\sigma_j}(0) \to \mathbb{R}^j\times Y_C^{N-j} \}$ with $O_Y\in v_i(B_{2\sigma}(0))$ such that $\Ord_{\phi,v_i}(0,R_0)\leq \Lambda$ and $\sup_{B_{\sigma_j}(0)}|\nabla u_i|\leq M$, and for $\sigma_i<\min \{\frac{1}{i}, R_0/256\},$
    \begin{itemize}
        \item [(i)]$F_i=\{x \in B_{\sigma_i}(0):\Ord_{\phi,v_i}(x,\sigma_i\rho)>\Gamma_i-\frac{1}{i}\} \ \rho\sigma_i$-effectively spans a $k$-dimensional affine space $V^i$ using the points $\{x^i_0,...,x^i_k\} \subset B_{\sigma_i}(0)$.
        \item [(ii)] $\Ord_{\phi,v_i}(x^i_j,\sigma_i\rho) >\Gamma_i-\frac{1}{i}$ for each $j \in \{0,...,k\},$ but there exists a point $y^i \in V^i \cap B_{\sigma_i}(0)$ such that $\Ord_{\phi,v_i}(y^i,2\sigma_i\rho^2)\leq \Gamma_i-\epsilon$
        \item[(iii)] $\displaystyle\sup_{x \in B_{2\sigma_i}(0)}\Ord_{\phi,v_i}(x,2\sigma_i)=\Gamma_i.$
    \end{itemize}
    By definition of $\Ord_{\phi,v_i}$, we have 
    \begin{equation*}
        \Gamma_i=\sup_{B_{2\sigma_i}(0)}e^{A(2\sigma_i)}(\BDOrd_{\phi,v_i}(x,2\sigma_i)+A(2\sigma_i)^2),
    \end{equation*}
    which implies that
    \begin{equation*}
        \sup_{B_{2\sigma_i}(0)}\BDOrd_{\phi,v_i}(x,2\sigma_i)=e^{-A(2\sigma_i)}\Gamma_i-A(2\sigma_i)^2:=\Gamma_i'.
    \end{equation*}
    Similarly, we have
    \begin{equation*}
        \Gamma_i-1/i<\Ord_{\phi,v_i}(x_j^i,\sigma_i\rho)=e^{A\sigma_i\rho}(\BDOrd_{\phi,v_i}(x_j^i,\sigma_i\rho)+A(\sigma_i\rho)^2),
    \end{equation*}
    which implies 
    \begin{align*}
        \BDOrd_{\phi,v_i}(x_j^i,\sigma_i\rho)&>e^{-A\sigma_i\rho}(\Gamma_i-1/i)-A(\rho\sigma_i)^2\\
        &=e^{A\sigma_i(2-\rho)}(\Gamma_i'+A(2\sigma_i)^2)-e^{-A\sigma_i\rho}/i-A(\rho\sigma_i)^2\\
        &>\Gamma_i'+A(2\sigma_i)^2-1/i-A(\rho\sigma_i)^2>\Gamma_i'-1/i.
    \end{align*}
    Then, since $\BDOrd_{\phi,v_i}(x_j^i,2\sigma_i)\leq \Gamma_i'$, we have that
    \begin{equation}\label{bd5.10:freqgap}
        \BDOrd_{\phi,v_i}(x_j^i,2\sigma_i)-\BDOrd_{\phi,v_i}(x_j^i,\sigma_i\rho)<1/i.
    \end{equation}
    Lastly, we have
    \begin{align*}
        \BDOrd_{\phi,v_i}(y^i,2\sigma_i\rho^2)&\leq e^{-2A\sigma_i\rho^2}(\Gamma_i-\eps)-A(2\sigma_i\rho^2)^2\\
        &=e^{A(2\sigma_i-2\sigma_i\rho^2)}(\Gamma_i'+A(2\sigma_i)^2)-e^{-2A\sigma_i\rho^2}\eps-A(2\sigma_i\rho^2)^2\\
        &=\Gamma_i'-\eps+(e^{A(2\sigma_i-2\sigma_i\rho^2)}-1)\Gamma_i'+(1-e^{-2A\sigma_i\rho^2})\eps\\
        &+e^{A(2\sigma_i-2\sigma_i\rho^2)}A((2\sigma_i)^2-(2\sigma_i\rho^2)^2).
    \end{align*}
    Then we have that $\BDOrd_{\phi,v_i}(y^i,2\sigma_i\rho^2)\leq \Gamma_i'-\eps+O(\sigma_i)$ for $i$ sufficiently large. Define $\tilde{v}_i: B_1(0) \to (Y_C^{N-j},d_i=\frac{1}{\lambda_i}d_{Y_C})$ where $\lambda_i=((256\sigma_i)^{1-n}I_{\phi,v_i}(0,256\sigma_i))^{\frac{1}{2}}$ and $\tilde{v}_i(x)=v_i(256\sigma_ix)$. Then, each $\tilde{v}_i$ satisfies:
    \begin{itemize}
        \item [(i)] $F_i=\{(256\sigma_i)^{-1}x \in B_{1/256}(0):\BDOrd_{\phi,\tilde{v}_i}((256\sigma_i)^{-1}x,\rho/256)>\Gamma_i'-\frac{1}{i}\} \ \frac{\rho}{256}$ -- effectively spans a $k$-dimensional affine space $V^i$ by using the points $\{(256\sigma_i)^{-1}x^i_0,...,(256\sigma_i)^{-1}x^i_k\} \subset B_{1/256}(0)$
        \item [(ii)] $\BDOrd_{\phi,\tilde{v}_i}((256\sigma_i)^{-1}x^i_j,\rho/256) >\Gamma_i'-\frac{1}{i}$ for each $j \in \{0,...,k\},$ but there exists a point $(256\sigma_i)^{-1}y^i \in V^i \cap B_{1/256}(0)$ such that $\BDOrd_{\phi,\tilde{v}_i}((256\sigma_i)^{-1}y^i,\rho^2/128)\leq \Gamma_i'-\epsilon+O(\sigma_i)$
        \item[(iii)]
        $\displaystyle\sup_{B_{1/128}(0)}\BDOrd_{\phi,\tilde{v}_i}\paren{(256\sigma_i)^{-1}x,\frac{1}{128}}=\Gamma_i'.$
    \end{itemize}
    
    By Lemma \ref{ineq:ord} with respect to $\tilde{v}_i,$ 
    we have the uniform bound $\Gamma_i'\leq \Gamma_i\leq C_{4.2}(\Lambda+1)$ for some constant $C_{4.2}$ for any $i.$ Then, analogous to the previous arguments, we may choose an (unrelabeled) subsequence such that:
    \begin{itemize}
        \item $\tilde{v}_i \to v_{\infty}$ locally uniformly in $B_{1/64}(0)$ in the sense of Lemma~\ref{compactness}, where $Y_{\infty}$ is a $F$-connected complex, such that the energy densities of the $\tilde{v}_i$'s converge weakly to that of $v_{\infty}$ on $B_{1/64}(0)$.
        \item $\Gamma_i' \to \Gamma'$ due to the uniform bound.
        \item $V^i \to V$ where $V$ is a $k$-dimensional affine subspace spanned by $\{x_0,...,x_k\} \subset \overline{B_{1/256}(0)}$ and $x_j=\lim_{i \to \infty}(256\sigma_i)^{-1}x^i_j$ for each $j$.
        \item $(256\sigma_i)^{-1}y^i \to y \in \overline{V \cap B_{1/256}(0)}$ such that 
        \begin{equation}\label{5.10:eq2}
            \BDOrd_{\phi,v_{\infty}}(y,\rho^2/128) \leq \Gamma'-\epsilon. 
        \end{equation}
    \end{itemize}
    From \eqref{bd5.10:freqgap}, $\BDOrd_{\phi,v_{\infty}}(x_j,1/128)-\BDOrd_{\phi,v_{\infty}}(x_j,\rho/256) \leq 0.$ Since $v_{\infty}$ is a harmonic map into a $F$-connected complex, the monotonicity property \cite[Proposition 4.1]{dees}
    implies $\BDOrd_{\phi,v_{\infty}}(x_j,1/128)-\BDOrd_{\phi,v_{\infty}}(x_j,\rho/256) =0.$ Then, \cite[Lemma 5.4]{bd} implies $v_{\infty}$ is homogeneous of some degree $\alpha_j \geq 1$ on $B_{1/4}(x_j)$ for $j=0,...,k.$ By \cite[Lemma 5.6]{bd}, $v_{\infty}$ is homogeneous of a fixed degree $\Gamma'$ about each $x_j,\ j \in\{0,...,k\},$ and then $k$-homogeneous about a $k$-plane $V.$ So for $y \in V, \ \BDOrd_{\phi,v_{\infty}}(y,\rho^2/128)= \Gamma',$ which contradicts (\ref{5.10:eq2}).    
\end{proof}
\section{Translating inhomogeneities}\label{sec:inhomogeneities}
This section is devoted to showing that if a point $x$ is pinched, i.e. the frequency pinching at $x$ is sufficiently small, then the inhomogeneity of $x$ is translated to the pinched points close to $x$. Before stating the main results of this section, we prove lemmas giving the control on the gradient of a homogeneous map approximating the singular component $v$. Then we use these estimates to show that if a point in the stratum is not approximately $(k+1)$-homogeneous, then nearby pinched points cannot be approximately $(k+1)$-homogeneous either.
\begin{lemma}\label{bd5.11}
    Suppose that $u:B_2(0) \to Y_C$ is a Lipschitz map into a conical $F$-connected complex. Let $h:\mathbb{R}^n \to Y_C$ be a $k$-homogeneous map of order $\alpha$ such that 
    \[\sup_{x \in B_{1}(0)}d(u,h) \leq C'.\]
    Then there exists a $k$-homogeneous map $\hat{h}:\mathbb{R}^n \to Y_C$ satisfying 
    \[\sup_{x \in B_{1}(0)} d(u,\hat{h}) \leq 2C'\]
    and such that 
    \[\sup_{x \in B_{1}(0)}|\nabla \hat{h}| \leq (1+\alpha)\sup_{B_{2}(0)}|\nabla u|.\]
\end{lemma}
\begin{proof}
The proof is identical to the proof of \cite[Lemma 5.11]{bd}.
\end{proof}

\begin{lemma}\label{bd5.12}
     Let $u$ be a harmonic map to an $N$-dimensional conical DM-complex $Y_C$ with splitting $u=(V,v):B_{\sigma_j}(0) \to (\mathbb{R}^j \times Y_C^{N-j},d_G)$ satisfying $\Ord_{\phi,v}(0,R_0)\leq \Lambda$ and $\sup_{B_{\sigma_j}(0)}|\nabla u| \leq M$, then there exists $\eta_1>0$ and $\calA\geq 1$, depending only on $\Lambda$, $n$, $M$, and the constant in \cite[(29)--(33)]{dm}, such that for $\sigma\in (0,\eta_1]$ and $\eta\in (0,\eta_1]$, if for $x\in B_{R_0/64}(0)$ there exists a homogeneous Lipschitz map $h:B_{\sigma}(x)\to Y_C$ such that
     \begin{equation*}
         \sup_{B_{\sigma}(x)}d(v,h)\leq \eta\paren{\sigma^{1-n}I_{\phi,v}(x,\sigma)}^{1/2},
     \end{equation*}
     then $\alpha\leq \calA$, where $\alpha$ is the degree of $h$.
\end{lemma}
\begin{proof}
    For $\sigma<R_0/256$, let $w:B_{8\sigma}(x) \to Y_C$ be the harmonic map with $w|_{\partial B_{8\sigma}(x)}=v|_{\partial B_{8\sigma}(x)}$. Although the authors claim \cite[Proposition 48]{dm} only in the case where $x\in \calS_j(u)$, the proof does not require this assumption. Thus, we may invoke \cite[Proposition 48]{dm} to assert that there exists a constant depending only on $n$, the Lipschitz constant of $u$, the constant in \cite[(29)-(33)]{dm}, and $\frac{8\sigma E^v(x,8\sigma)}{I^v(x,8\sigma)}$, such that $\sup_{B_{\sigma}(x)}d(v,w)\leq C\sigma\paren{(8\sigma)^{1-n}I^v(x,8\sigma)}^{1/2}$. Since ${I^v}'(x,r)\geq 0$ for all $r<R_0$, we have 
    \begin{equation}\label{eq:5.12.1}
        I_{\phi,v}(x,\sigma)\leq C(n)I^v(x,8\sigma),
    \end{equation}
    so by the triangle inequality we have $\sup_{B_{\sigma}(x)}d(w,h)\leq C(\eta+\sigma)\paren{(8\sigma)^{1-n}I^v(x,8\sigma)}^{1/2}.$ By \cite[Lemma 3.1]{dees}, there exists a constant $C$ depending on $\frac{8\sigma E^w(x,8\sigma)}{I^w(x,8\sigma)}$ such that 
    \begin{equation}\label{eqn:1}
        I^v(x,8\sigma)=I^w(x,8\sigma)\leq CI^w(x,\sigma/2)\leq CI_{\phi,w}(x,\sigma). 
    \end{equation}
    Since $E^w(x,8\sigma)\leq E^v(x,8\sigma)$ and $I^w(x,8\sigma)=I^v(x,8\sigma)$, we have 
    \begin{equation}\label{eq:5.12.2}
        \frac{8\sigma E^w(x,8\sigma)}{I^w(x,8\sigma)}\leq \frac{8\sigma E^v(x,8\sigma)}{I^v(x,8\sigma)}.
    \end{equation}
    Thus, there exists a constant $C$ depending only on $\frac{8\sigma E^v(x,8\sigma)}{I^v(x,8\sigma)}$ such that 
    \begin{equation*}
        \sup_{B_{\sigma}}d(w,h)\leq C(\eta+\sigma)\paren{\sigma^{1-n}I_{\phi,w}(x,\sigma)}^{1/2}.
    \end{equation*}
    Then, by \cite[Lemma 5.12]{bd}, for $\sigma,\eta\leq \min\{R_0/256,\eta_0(n)/2C\}:=\eta_1$ for $C$ depending on $n$, the Lipschitz constant of $u$, and the constant in \cite[(29)--(32)]{dm} and $\frac{8\sigma E^v(x,8\sigma)}{I^v(x,8\sigma)}$, we have $\alpha\leq A$, where $A$ depends on $n$, the Lipschitz constant of $u$, and the constant in \cite[(29)--(32)]{dm}, $\frac{8\sigma E^v(x,8\sigma)}{I^v(x,8\sigma)}$, and $\frac{
    8\sigma E_{\phi,w}(x,8\sigma)}{I_{\phi,w}(x,8\sigma)}$. We have that $E_{\phi,w}(x,8\sigma)\leq E^w(x,8\sigma)\leq E^v(x,8\sigma)$. We also have $I^v(x,8\sigma)=I^w(x,8\sigma)\leq CI^w(x,4\sigma)\leq CI_{\phi,w}(x,8\sigma)$ for $C$ depending only on an upper bound for $\frac{8\sigma E^w(x,8\sigma)}{I^w(x,8\sigma)}$, and thus only on $\frac{8\sigma E^v(x,8\sigma)}{I^v(x,8\sigma)}$. Thus, we have that $\frac{
    8\sigma E_{\phi,w}(x,8\sigma)}{I_{\phi,w}(8\sigma)}$ is bounded by a constant depending only on $\frac{8\sigma E^v(x,8\sigma)}{I^v(x,8\sigma)}$. We have $E^v(x,8\sigma)\leq E_{\phi,v}(x,16\sigma)$, and by \eqref{eq:5.12.1} and \eqref{eqn:6}, there exists $C(\Lambda)$ such that $I^v(x,8\sigma)\geq C(n)I_{\phi,v}(x,8\sigma)\geq C(\Lambda)I_{\phi,v}(x,16\sigma)$. Therefore, we have $\frac{8\sigma E^v(x,8\sigma)}{I^v(x,8\sigma)}\leq C(\Lambda)\BDOrd_{\phi,v}(x,16\sigma)\leq C(\Lambda)\Ord_{\phi,v}(x,16\sigma)\leq C(\Lambda)$, where the last inequality is by Lemma \ref{ineq:ord}. Thus, $\eta_1$ and $A$ depend only on $n$, the Lipschitz constant of $u$, and the constant in \cite[(29)--(32)]{dm} and $\Lambda$.
\end{proof}

\begin{corollary}\label{cor1}
    Let $u$ be a harmonic map to an $N$-dimensional conical DM-complex $Y_C$ with splitting $u=(V,v):B_{\sigma_j}(0) \to (\mathbb{R}^j \times Y_C^{N-j},d_G)$ satisfying $\Ord_{\phi,v}(0,R_0)\leq \Lambda$ and $\sup_{B_{\sigma_j}(0)}|\nabla u| \leq M$, then there exist $\eta_1>0$ and $C>0$ depending only on $\Lambda$, $n$, $M$, and the constant in \cite[(29)--(32)]{dm}, such that if for $\sigma,\eta\in (0,\eta_1]$ and for $x\in B_{R_0/64}(0)$ there exists a $k$-homogeneous Lipschitz map $h:B_{1}(x)\to (Y_C,\lambda_{\sigma}^{-1}d)$ with $\lambda_{\sigma}=(\sigma^{1-n}I_{\phi,v}(x,\sigma))^{\frac{1}{2}}$ such that for $v_{\sigma}:B_1(x) \to (Y_C,\lambda_{\sigma}^{-1}d)$ with $v_{\sigma}(y)=v(\sigma (y-x)+x)$,
    \[\sup_{B_{1}(x)}d(v_{\sigma},h)\leq \eta\]
    then there exists another $k$-homogeneous map $\hat{h}:B_{1}(x)\to (Y_C,\lambda_{\sigma}^{-1}d)$ such that
    \[\sup_{B_{1}(x)}d(v_{\sigma},\hat{h})\leq C(\sigma+\eta)\]
    and such that
    \[\sup_{B_{1}(x)}|\nabla \hat{h}|\leq C.\]
\end{corollary}
\begin{proof}
    For $\sigma,\eta \leq \eta_1$, where $\eta_1$ is as in Lemma \ref{bd5.12}, we have $\alpha\leq A$, where $\alpha$ is the degree of $h$. Also, as shown in the proof of Lemma \ref{bd5.12}, for $\sigma,\eta\leq \eta_1$ we have
    \begin{equation*}
        \sup_{B_{1}(x)}d(w_{\sigma},h)\leq C(\eta+\sigma),
    \end{equation*}
    where $w_{\sigma}:B_8(x)\to (Y_C,\lambda_{\sigma}^{-1}d)$ is the harmonic map with $w_{\sigma}|_{\partial B_8(x)}=v_{\sigma}|_{\partial B_8(x)}$. Then, by Lemma \ref{bd5.11}, there exists a $k$-homogeneous map $\hat{h}$ of degree $\alpha$ such that
    \begin{equation*}
        \sup_{B_{1}(x)}d(w_{\sigma},\hat{h})\leq 2C(\sigma+\eta),
    \end{equation*}
    and
    \begin{equation*}
        \sup_{B_{1}(x)}|\nabla \hat{h}|\leq (1+\alpha)\sup_{B_{2}(x)}|\nabla w_{\sigma}|\leq C(n)(1+A)E^{w_{\sigma}}(x,8)\leq CE^{v_{\sigma}}(x,8).
    \end{equation*}
    By (\ref{eqn:6}) and Lemma \ref{ineq:ord}, we have
    \begin{equation*}
        E^{v_{\sigma}}(x,8)=\frac{\sigma E^v(x,8\sigma)}{I_{\phi,v}(x,\sigma)}\leq \frac{16\sigma E_{\phi,v}(x,16\sigma)}{I_{\phi,v}(x,\sigma)}\leq C(\Lambda)\BDOrd_{\phi,v}(x,16\sigma)\leq C(\Lambda).
    \end{equation*}
    By \cite[Proposition 48]{dm}, we have $\sup_{B_{1}(x)}d(v_{\sigma},w_{\sigma})\leq C\sigma^2$, so we have
    \begin{equation*}
        \sup_{B_{1}(x)}d(v_{\sigma},\hat{h})\leq 2C(\sigma+\eta)+C\sigma^2\leq C(\sigma+\eta).
    \end{equation*}
\end{proof}
\begin{lemma}\label{bd5.13}
    Let $u$ be a harmonic map to an $N$-dimensional conical DM-complex $Y_C$ with splitting $u=(V,v):B_{\sigma_j}(0) \to (\mathbb{R}^j \times Y_C^{N-j},d_G)$ satisfying $\Ord_{\phi,v}(0,R_0)\leq \Lambda \text{ and } \sup_{B_{\sigma_j}(0)}|\nabla u| \leq M.$ Let $0 < \rho < 1$ and $0<\eta<\eta_1$ where $\eta_1$ is chosen in the sense of Lemma \ref{bd5.12}. Then, there exist $\sigma_{\ref{bd5.13}},\delta_{\ref{bd5.13}}=(\sigma_{\ref{bd5.13}},\delta_{\ref{bd5.13}})(n,\rho,\eta,M,X,\Lambda)>0$ such that if $0<\sigma<\sigma_{\ref{bd5.13}}$ and $0<\delta<\delta_{\ref{bd5.13}}$ and
    \begin{itemize}
        \item the set $F:=\{x \in \mathcal{S}^k_{0,\eta,\delta\sigma}(v) \cap B_{\sigma}(0): W^{8\sigma}_{\sigma}(x,v) <\delta\} \ \rho\sigma$-effectively spans a $k$-dimensional affine space $V,$
        \item there exists $y \in B_{\sigma}(0) \cap V$ with $W^{8\sigma}_{\sigma}(y,v)<\delta,$
    \end{itemize}
    then $v$ is not $(\theta\eta,t,k+1)$-homogeneous at $y$ for all $t \in [2\rho\sigma,2\sigma]$, where $\theta=\frac{1}{2C}$ and $C$ is as in Corollary \ref{cor1}.
\end{lemma}
\begin{proof}
    Suppose no such $\sigma_{\ref{bd5.13}},\delta_{\ref{bd5.13}}>0$ exist. Then there exists a sequence of harmonic maps $u_i=(V_i,v_i):B_{\sigma_j}(0) \to \mathbb{R}^j \times Y_C^{N-j}$ such that $\Ord_{\phi,v_i}(0,R_0)\leq \Lambda$ and $\sup_{B_{\sigma_j}(0)}|\nabla u_i|\leq M$, and such that for $\sigma_i < \min\{\frac{1}{i},R_0/576,\eta_1\},$
    \begin{itemize}
        \item the set $F_i:=\{x \in \mathcal{S}^k_{0,\eta,\sigma_i/i}(v_i) \cap B_{\sigma_i}(0):W^{8\sigma_i}_{\sigma_i}(x,v_i) <1/i\} \ \rho\sigma_i$-effectively spans a $k$-dimensional affine space $V^i$ by points $\{x_0^i,...,x_k^i\} \subset \mathcal{S}_{0,\eta,\sigma_i/i}(v_i)\cap B_{\sigma_i}(0),$ where $\eta<\eta_1$ is defined as in Lemma \ref{bd5.12}
        \item there exists a sequence of $y^i \in B_{\sigma_i}(0) \cap V^i$ with $W_{\sigma_i}^{8\sigma_i}(y^i,v_i)<\frac{1}{i}$
        \item there exists a sequence of scales $t_i \in [2\rho\sigma_i,2\sigma_i]$ so that $v_i$ is $(\theta\eta,t_i,k+1)$-homogeneous at $y^i$.
    \end{itemize}
    Define blow-up maps $\tilde{v}_i:B_1(0) \to (Y_C^{N-j},\frac{1}{\lambda_i}d)$ where $\lambda_i=((576\sigma_i)^{1-n}I_{\phi,v_i}(0,576\sigma_i))^{\frac{1}{2}}$ by $\tilde{v}_i(x)=v_i(576\sigma_ix).$ Since each $x_j^i$ has $\tilde{v}_i((576\sigma_i)^{-1}x_j^i)=O_Y$, the requirements of Lemma \ref{compactness} are met. By Lemma \ref{compactness} a subsequence of $\tilde{v}_i$ converges locally uniformly to a harmonic map $v_{\infty}:B_{1/64}(0) \to Y_{\infty}$ in the pullback sense. Since $\Ord_{\phi,v_i}(x_j^i,r):=e^{Ar}(\BDOrd_{\phi,v_i}(x_j^i,r)+Ar^2)$ for all $r>0$ and $r\mapsto \Ord_{\phi,v_i}(x_j^i,r)$ is monotone, we have for any $r_1\leq r_2$,
    \begin{align*}
        |\BDOrd_{\phi,v_i}(x_j^i,r_2)-\BDOrd_{\phi,v_i}(x_j^i,r_1)|=&|e^{-Ar_2}\Ord_{\phi,v_i}(x_j^i,r_2)-Ar_2^2
        -e^{-Ar_1}\Ord_{\phi,v_i}(x_j^i,r_1)+Ar_1^2|\\
        =&|e^{-Ar_2}W_{r_1}^{r_2}(x_j^i,v_i)+(e^{-Ar_2}-e^{-Ar_1})\Ord_{\phi,v_i}(x_j^i,r_1)-C_1'(r_2^2-r_1^2)|\\
        \leq& W_{r_1}^{r_2}(x_j^i,v_i)+A(r_2^2-r_1^2)+A(r_2-r_1)\Ord_{\phi,v_i}(x_j^i,r_1).
    \end{align*}
    Thus as $\sigma_i \to 0,$ $W^{8\sigma_i}_{\sigma_i}(\cdot,v_i) \to 0$ implies $|\BDOrd_{\phi,v_i}(\cdot,8\sigma_i)-\BDOrd_{\phi,v_i}(\cdot,\sigma_i)| \to 0.$  Then for the sequence $\{\tilde{v}_i\},$
    \begin{itemize}
        \item $\BDOrd_{\phi,\tilde{v}_i}((576\sigma_i)^{-1}x_j^i,1/72)-\BDOrd_{\phi,\tilde{v}_i}((576\sigma_i)^{-1}x_j^i,1/576) \to 0$ as $i \to \infty$ for \\ $\{(576\sigma_i)^{-1}x_0^i,...,(576\sigma_i)^{-1}x_k^i\} \subset \mathcal{S}_{0,\eta,1/576i}(\tilde{v}_i)\cap B_{1/576}(0),$
        \item $\BDOrd_{\phi,\tilde{v}_i}((576\sigma_i)^{-1}y^i,1/72)-\BDOrd_{\phi,\tilde{v}_i}((576\sigma_i)^{-1}y^i,1/576) \to 0$ as $i \to \infty$ for $(576\sigma_i)^{-1}y^i \in B_{1/576}(0) \cap V^i$  
        \item a sequence of scales $(576\sigma_i)^{-1}t_i \in [\frac{\rho}{288},\frac{1}{288}]$ so that $\tilde{v}_i$ is $(\theta\eta,(576\sigma_i)^{-1}t_i,k+1)$-homogeneous at $(576\sigma_i)^{-1}y^i.$
    \end{itemize}
    In particular, there is a sequence of $(k+1)$-homogeneous maps $\tilde{h}_i$ such that 
    \[
    \sup_{B_{(576\sigma_i)^{-1}t_i}(\sigma_i^{-1}y^i)}d(\tilde{v}_i,\tilde{h}_i^{(576\sigma_i)^{-1}y^i})\leq \theta\eta\left(\frac{I_{\phi,\tilde{v}_i}(0,(576\sigma_i)^{-1}t_i)}{((576\sigma_i)^{-1}t_i)^{n-1}}\right)^{\frac{1}{2}}.
    \]
    Following Corollary \ref{cor1}, we obtain a new sequence of $(k+1)$-homogeneous maps $h_i$ with a uniform gradient bound $C'$ such that for each $i,$
    \begin{equation*}
        \sup_{B_{(576\sigma_i)^{-1}t_i}((576\sigma_i)^{-1}y^i)}d(\tilde{v}_i,h_i^{(576\sigma_i)^{-1}y^i})\leq C(\theta\eta+(576\sigma_i))\left(\frac{I_{\phi,\tilde{v}_i}(0,(576\sigma_i)^{-1}t_i)}{((576\sigma_i)^{-1}t_i)^{n-1}}\right)^{\frac{1}{2}}
    \end{equation*}
   and 
    \[
    \sup_{B_{(576\sigma_i)^{-1}t_i}((576\sigma_i)^{-1}y^i)}\big|\nabla h_i^{(576\sigma_i)^{-1}y^i}\big| \leq C.
    \]

By Lemma \ref{compactness}, there exists an unrelabeled subsequence $\tilde{v}_i$ converging to $v_{\infty}$ uniformly on $B_{1/64}(0)$ such that the energy densitites of the $\tilde{v}_i$'s converge weakly to that of $v_{\infty}$. By taking a further subsequence if necessary, we may assume $(576\sigma_i)^{-1}t_i\to t\in [\frac{\rho}{288},\frac{1}{288}]$ and $(576\sigma_i)^{-1}y^i\to y \in\overline{B_{1/576}(0)}$ as $i\to\infty$.  Also, the corresponding gradient bounds allow us to find a subsequence of $h_i$ converging uniformly to a $(k+1)$-homogeneous map $h.$ By homogeneity, this uniform convergence can be extended to $B_{1/288}(y)$. Therefore, we have
    \begin{itemize}
        \item [(i)] $\{(576\sigma_i)^{-1}x_0^i,...,(576\sigma_i)^{-1}x_k^i\}$ converges to $\{x_0,...,x_k\} \subset \overline{B_{1/576}(0)}$ and \\ $\BDOrd_{\phi,v_{\infty}}(x_j,1/72)-\BDOrd_{\phi,v_{\infty}}(x_j,1/576)=0$ for $j=0,...,k$,
        \item [(ii)] $V^i \to V$ which is spanned by $\{x_0,...,x_k\},$
        \item [(iii)] $(576\sigma_i)^{-1}y^i$ converges to $y \in \overline{B_{1/576}(0)}\cap V$ and $\BDOrd_{\phi,v_{\infty}}(y,1/72)-\BDOrd_{\phi,v_{\infty}}(y,1/576) = 0,$
        \item [(iv)] $(576\sigma_i)^{-1}t_i$ tends to $t \in [\frac{\rho}{288},\frac{1}{288}],$
        \item [(v)] $h_i \to h$ uniformly in the pullback sense where $h$ is a $(k+1)$-homogeneous map with
    \end{itemize}
\[
\sup_{B_t(y)}d_{\infty}(v_{\infty},h^y) \leq C\theta\eta \left(\frac{I_{\phi,v_{\infty}}(0,t)}{t^{n-1}}\right)^{\frac{1}{2}}=\frac{\eta}{2}\left(\frac{I_{\phi,v_{\infty}}(0,t)}{t^{n-1}}\right)^{\frac{1}{2}},
\]
where $h^y=h(x-y).$ Since $v_{\infty}$ is homogeneous about $V$ from \cite[Lemmas 5.4 and 5.6]{bd} and $x_j, \ y \in V$ for $0 \leq j \leq k,$
\[
\sup_{B_t(x_j)}d_{\infty}(v_{\infty},h^{x_j}) \leq \frac{\eta}{2}\left(\frac{I_{\phi,v_{\infty}}(0,t)}{t^{n-1}}\right)^{\frac{1}{2}}.
\]
For sufficiently large $i$ this contradicts $(576\sigma_i)^{-1}x_j^i 
\in \mathcal{S}^k_{0,\eta,1/576i}(\tilde{v}_i)$ whenever $1/576i < t.$
\end{proof}

\begin{lemma}\label{bd5.14}
    Let $u$ be a harmonic map to an $N$-dimensional conical DM-complex $Y_C$ with splitting $u=(V,v):B_{\sigma_j}(0) \to (\mathbb{R}^j \times Y_C^{N-j},d_G)$ satisfying $\Ord_{\phi,v}(0,R_0)\leq \Lambda$ and $ \sup_{B_{\sigma_j}(0)}|\nabla u| \leq M.$ Given $0 < \rho <2$ and $0 <\eta <\eta_1$ where $\eta_1$ is chosen in the sense of Lemma \ref{bd5.12}, there exist $ \sigma_{\ref{bd5.14}}, \, \delta_{\ref{bd5.14}}=(\sigma_{\ref{bd5.14}}, \, \delta_{\ref{bd5.14}})(n,\rho,\eta,M,X,\Lambda)>0$ such that if $0<\sigma<\sigma_{\ref{bd5.14}}$ and $O_Y\in v(B_{\sigma}(0))$ and $W^{4\sigma}_{\rho \sigma}(0,v)<\delta_{\ref{bd5.14}}$ and there exists $y \in B_{\sigma}(0)$ such that
    \begin{itemize}
        \item $W^{4\sigma}_{\rho\sigma}(y,v)<\delta_{\ref{bd5.14}}$
        \item $v$ is not $(\eta,\sigma,k+1)$-homogeneous about $y$.
    \end{itemize}
    Then $v$ is not $(\theta\eta,\sigma,k+1)$-homogeneous about $0$, where $\theta$ is as in Lemma \ref{bd5.13}.
\end{lemma}
\begin{proof}
Suppose no such $\sigma_{\ref{bd5.14}}, \, \delta_{\ref{bd5.14}}>0$ exist. Then there exists a sequence of harmonic maps $u_i=(V_i,v_i):B_{\sigma_j}(0) \to \mathbb{R}^j \times Y_C^{N-j}$ with $O_Y\in v_i(B_{\sigma_i}(0))$ such that $\Ord_{\phi,v_i}(0,R_0)\leq \Lambda$ and $\sup_{B_{\sigma_j}(0)}|\nabla u_i|\leq M$, and such that for $\sigma_i < \min\{\frac{1}{i},R_0/320,\eta_1\},$
    \begin{itemize}
        \item $W_{\rho\sigma_i}^{4\sigma_i}(0,v_i)<\frac{1}{i},$
        \item there exists $y^i \in B_{\sigma_i}(0)$ such that $W_{\rho\sigma_i}^{4\sigma_i}(y^i,v_i)<\frac{1}{i},$
        \item $v_i$ is not $(\eta,\sigma_i,k+1)$-homogeneous about $y^i,$
        \item $v_i$ is $(\theta\eta,\sigma_i,k+1)$-homogeneous about $0.$
    \end{itemize}
    Define blow-up maps $\tilde{v}_i:B_1(0) \to (Y_C^{N-j},\frac{1}{\lambda_i}d)$ with $\lambda_i=((320\sigma_i)^{1-n}I_{\phi,v_i}(0,320\sigma_i))^{\frac{1}{2}}$ by $\tilde{v}_i(x)=v_i(320\sigma_ix).$ For $\{\tilde{v}_i\},$
    \begin{itemize}
        \item as $\sigma_i \to 0, W_{\rho/320}^{1/80}(0,\tilde{v}_i) \to 0,$ 
        \item $W_{\rho/320}^{1/80}((320\sigma_i)^{-1}y^i,\tilde{v}_i) \to 0$ as $\sigma_i \to 0,$
        \item $\tilde{v}_i$ is not $(\eta,1/320,k+1)$-homogeneous about $(320\sigma_i)^{-1}y^i,$
        \item $\tilde{v}_i$ is $(\theta\eta,1/320,k+1)$-homogeneous about $0$ with homogeneous map $\tilde{h}_i.$
    \end{itemize}
    By Lemma \ref{compactness}, $\tilde{v}_i$ converges uniformly (up to a subsequence) to a harmonic map $v_{\infty}:B_{1/64}(0) \to Y_{\infty}$ such that the energy densities of the $\tilde{v}_i's$ converge weakly on $B_{1/64}(0)$ to the energy density of $v_{\infty}$. Following the proof of Lemma \ref{bd5.13}, we can find a subsequence of $(320\sigma_i)^{-1}y^i$ converging to a point $y \in\overline{B_{1/320}(0)}$ and replace $\{\tilde{h}_i\}$ by a subsequence of $h_i$ converging uniformly to a $(k+1)$-homogeneous map $h.$ Now we have 
    \begin{itemize}
        \item [(i)] $\BDOrd_{\phi,\tilde{v}_i}(0,\frac{1}{80})-\BDOrd_{\phi,\tilde{v}_i}(0,\frac{\rho}{320}) \to 0,$
        \item [(ii)] $\BDOrd_{\phi,\tilde{v}_i}((320\sigma_i)^{-1}y^i,\frac{1}{80})-\BDOrd_{\phi,\tilde{v}_i}((320\sigma_i)^{-1}y^i,\frac{\rho}{320}) \to 0$ for $(320\sigma_i)^{-1}y^i \in B_{1/320}(0)$ converging to $y \in \overline{B_{1/320}(0)},$
        \item [(iii)] $v_{\infty}$ is not $(\eta,1/320,k+1)$-homogeneous about $y,$
        \item [(iv)] $v_{\infty}$ is $(C\theta\eta,1/320,k+1)$-homogeneous about $0$ with
    \end{itemize}
    \begin{equation}\label{5.14:eq1}
    \sup_{B_{1/320}(0)}d_{\infty}(v_{\infty},h) \leq C\theta\eta\left(\frac{I_{\phi,v_{\infty}}(0,1/320)}{(1/320)^{n-1}}\right)^{\frac{1}{2}}.
    \end{equation}
    By \cite[Lemma 5.4]{bd}, $v_{\infty}$ is homogeneous about both $0$ and $y.$ In particular, if $y\neq0,$ by \cite[Lemma 5.6]{bd}, $v_{\infty}$ is invariant along the line between $0$ and $y.$ Hence,
    \begin{equation}\label{5.14:eq2}
    \sup_{B_{1/320}(y)}d_{\infty}(v_{\infty},h^y) \leq C\theta\eta \left(\frac{I_{\phi,v_{\infty}}(y,1/320)}{(1/320)^{n-1}}\right)^{\frac{1}{2}}.
    \end{equation}
    Since $(320\sigma_i)^{-1}y^i$ converges to $y \in \overline{B_{1/320}(0)}$ and $\tilde{v}_i$ converges to $v_{\infty}$ locally uniformly on $B_{1/64}(0),$ for large sufficiently $i,$ (\ref{5.14:eq2}) contradicts the assumption that $\tilde{v}_i$ is not $(\eta,\frac{1}{320},k+1)$-homogeneous about $(320\sigma_i)^{-1}y^i.$ If $y=0,$ (\ref{5.14:eq1}) is still a contradiction.
\end{proof}

\section{Mean Flatness}\label{sec:meanflatness}
In this section, we firstly define the $k$-th mean flatness of a measure and recall the rectifiability result in \cite{at}. Then, we prove that if $v$ is not $(\eta,\sigma,k+1)$-homogeneous at $x$ and the frequency pinching is bounded, then we can control the $k$-th mean flatness by the frequency pinching. This provides a link between the monotonicity formulae and the flatness estimates needed for the covering arguments in Appendix.

\begin{definition}
    Let $\mu$ be a Radon measure on $\RR^n$ and let $k\in \{0,1,\ldots,n-1\}$. The $k$\textbf{-th mean flatness} of $\mu$ on $B_r(x)\subset\RR^n$ is defined as
    \[D^k_{\mu}(x,r):=\inf_Lr^{-k-2}\int_{B_r(x)}\dist(y,L)^2d\mu(y),\]
    where the infimum is over all affine $k$-planes $L\subset\RR^n$.
\end{definition}
The $k$-th mean flatness of $\mu$ on $B_r(x)$ can be computed using the formula $D_{\mu}^k(x,r)=r^{-k-2}\sum_{\ell=k+1}^n\lambda_{\ell}$, where $0\leq \lambda_n\leq \lambda_{n-1}\leq\cdots\leq \lambda_1$ are the eigenvalues of the bilinear form $b(v,w):=\int_{B_{r}(x)}((x-\bar{x})\cdot v)((x-\bar{x})\cdot w)$ (see \cite[Definition 6.1]{dees}). The $k$-th mean flatness of a measure provides a necessary and sufficient condition for $k$-rectifiability of a set, which we will use to show $k$-rectifiability of $\calS^k_{0,\eta}(v)$.
\begin{lemma}[Corollary 1.3. in \cite{at}]\label{rectmeanflatness}
    If $E\subset \RR^n$ is an $\calH^k$-measurable set with $\calH^k(E)<\infty$, then $E$ is $k$-rectifiable if and only if $\int_0^1D_{\mu}^k(x,s)\frac{ds}{s}<\infty$ for $\mu$-a.e. $x$, where $\mu:=\calH^k\resmes E$.
\end{lemma}
We will use the frequency pinching of $v$ to bound the $k$-th mean flatness of $\calS_{0,\eta}^k(v)$. To do this, we first need two lemmas:
\begin{lemma}[Lemma 6.3 in \cite{bd}]\label{bd6.3}
    Let $x_0,...,x_n$ be orthonormal coordinates for $\mathbb{R}^n$ and let $u:\Omega \to X$ be a harmonic map from a convex domain into an $F$-connected complex. Suppose that $B_r(x_0)\subset \Omega$ and $u|_{B_r(x_0)}$ is a function depending only on $x_1,...,x_j$ for $j<n.$ Then, $u$ is a function of only $x_0,...,x_j$ on all of $\Omega.$
\end{lemma}
\begin{lemma}\label{toolformainflatness}
    Let $u$ be a harmonic map to an $N$-dimensional conical DM-complex $Y_C$ with splitting $u=(V,v):B_{\sigma_j}(0) \to (\mathbb{R}^j \times Y_C^{N-j},d_G)$ satisfying $\Ord_{\phi,v}(0,R_0) \leq \Lambda,$ and $ \sup_{B_{\sigma_j}(0)}|\nabla u| \leq M,$ where $R_0$ is as in Lemma \ref{monotoneorder}. Then for any $\eta>0, \rho \in (0,\frac{1}{2}),$ there are constants $\sigma_{\ref{toolformainflatness}}, \delta_{\ref{toolformainflatness}}=(\sigma_{\ref{toolformainflatness}}, \delta_{\ref{toolformainflatness}})(n,\rho,\eta,M,X,\Lambda)>0$ so that the following holds for all $0<\sigma<\sigma_{\ref{toolformainflatness}}$: Suppose that $O_Y\in v(B_{2\sigma}(0))$, $v$ is not $(\eta,\sigma,k+1)$-homogeneous at $0$ and $W^{2\sigma}_{\rho\sigma}(0,v)<\delta_{\ref{toolformainflatness}}$. Then for any orthonormal set $\{s_1,...,s_{k+1}\},$ 
    \[
    \sigma^{2-n}\int_{B_{5\sigma/4}(0)\setminus B_{3\sigma/4}(0)} \sum^{k+1}_{j=1}|\partial_{s_j}v|^2\geq \delta_{\ref{toolformainflatness}}.
    \]
\end{lemma}
\begin{proof}
    Fix $\eta>0$ and $\rho \in (0,1)$ and suppose that the claimed inequality does not hold for any $\delta, \sigma>0.$ Then there exists a sequence $\{u_i=(V_i,v_i):B_{\sigma_j}(0),\to \mathbb{R}^j \times Y_C^{N-j}\}$ with $\Ord_{\phi,v_i}(0,R_0)\leq \Lambda$ and $\sup_{B_{\sigma_j}(0)}|\nabla u_i|\leq M$ such that for $\sigma_i <\min \{\frac{1}{i},R_0/128\}, \ v_i$ is not $(\eta,\sigma_i,k+1)$-homogeneous at $0$, $O_Y\in v_i(B_{2\sigma_i}(0))$, and $W^{2\sigma_i}_{\rho\sigma_i}(0,v_i)<\frac{1}{i}.$ Additionally, there exists an orthonormal set $\{s_1^i,...,s_{k+1}^i\}$ so that 
    \[
    \sigma_i^{2-n} \int_{B_{5\sigma_i/4}(0)\setminus B_{3\sigma_i/4}(0)} \sum_{i=1}^{k+1} |\partial_{s^i_j}v_i|^2 <\frac{1}{i}.
    \]
    Then the blow-up maps $\tilde{v}_i:B_1(0) \to (Y^{N-j}_C,d_i=\frac{1}{\lambda_i}d)$ where \\ $\lambda_i=((128\sigma_i)^{1-n}I_{\phi,v_i}(0,128\sigma_i))^{\frac{1}{2}}$ and $\tilde{v}_i(x)=v_i(128\sigma_i x)$ satisfy:
    \begin{itemize}
        \item $\tilde{v}_i$ is not $(\eta,\frac{1}{128},k+1)$-homogeneous at $0$
        \item $\BDOrd_{\phi,\tilde{v}_i}(0,\frac{1}{64})-\BDOrd_{\phi,\tilde{v}_i}(0,\frac{\rho}{128})<\frac{1}{i}$
        \item $\displaystyle\int_{B_{5/512}(0)\setminus B_{3/512}(0)} \sum_{j=1}^{k+1}|\partial_{s_j^i}\tilde{v}_i|^2<\frac{1}{i}$ for some orthonormal set $\{s_1^i,...,s_{k+1}^i\}.$
    \end{itemize}
    By Lemma \ref{compactness}, an unrelabeled subsequence $\tilde{v}_i \to v_{\infty}$ uniformly on $B_{1/64}(0)$ in the pullback sense, where $v_{\infty}:B_{1/64}(0) \to Y_{\infty}$ is a harmonic map and the (directional) energy density measures of the $\tilde{v}_i$'s converge weakly to those of $v_{\infty}$. By choosing a further subsequence if necessary, we have
    \begin{itemize}
        \item [(i)]$\BDOrd_{\phi,v_{\infty}}(0,\frac{1}{64})-\BDOrd_{\phi,v_{\infty}}(0,\frac{\rho}{128})=0,$
        \item[(ii)] $\displaystyle\int_{B_{5/512}(0)\setminus B_{3/512}(0)} \sum_{j=1}^{k+1}|\partial_{s_j}v_{\infty}|^2=0$ where $\{s_1^i,...,s_{k+1}^i\}$ converges to $\{s_1,...,s_{k+1}\}.$
    \end{itemize}
Statement (i) implies $v_{\infty}$ is homogeneous with respect to $0$ on $B_{1/64}(0)$ from \cite[Lemma 5.4]{bd}. Statement (ii) implies that $v_{\infty}|_{B_{r}(y)}$ does not depend on the directions $s_1,...,s_{k+1}$ for any $B_{r}(y) \subset B_{5/512}(0)\setminus B_{3/512}(0)$ since $\partial_{s_j}v_{\infty}\equiv 0$ a.e. for any $i=1,...,k+1.$ By Lemma \ref{bd6.3}, $v_{\infty}$ is independent of $s_1,...,s_{k+1}$ everywhere on $B_{1/64}(0).$ So since $v_{\infty}$ is homogeneous about $0$ and is invariant with respect to the subspace spanned by $\{s_1,...,s_{k+1}\},$ we know $v_{\infty}$ is $(k+1)$-homogeneous at $0.$  This contradicts that $\tilde{v}_i$ is not $(\eta, 1/128,k+1)$-homogeneous at $0$ for sufficiently large $i.$
\end{proof}

\begin{lemma}\label{bd6.2}
    Let $u$ be a harmonic map to an $N$-dimensional conical DM-complex $Y_C$ with splitting $u=(V,v):B_{\sigma_j}(0) \to (\mathbb{R}^j \times Y_C^{N-j},d_G)$ satisfying $\Ord_{\phi,v}(0,R_0)\leq \Lambda,$ and  $\sup_{B_{\sigma_j}(0)}|\nabla u| \leq M.$ For any $\eta>0, \ \rho\in(0,1),$ there is a constant $C=C(n,\rho,\eta,M,Y_C,\Lambda)$ for which the following holds. 

    Let $\sigma_{\ref{bd6.2}}=\min\{R_1,\sigma_{\ref{toolformainflatness}}\}$, and assume $0<\sigma<\sigma_{\ref{bd6.2}}$. Suppose that $\mu$ is a finite non-negative Radon measure and $v$ is not $(\eta,\sigma,k+1)$-homogeneous at $x \in B_{R_0/64}(0)$. If $W_{\rho\sigma}^{2\sigma}(x,v)<\delta_{\ref{toolformainflatness}}$ and $O_Y\in v(B_{2\sigma}(x))$, then 
    \[D_{\mu}^k(x,\sigma/8) \leq \frac{C
    \sigma^{1-n}I_{\phi,v}(x,\sigma)}{\sigma^k} \paren{\int_{B_{\sigma/8}(x)}W_{\sigma/8}^{4\sigma}(y) \, d\mu(y)+\sigma}.\]
\end{lemma}

\begin{proof}
    Let $x=0$ by translation. Let $\bar{x}$ be the barycenter of $\mu$ in $B_{\sigma/8}(0)$ and let $\{s_1,...,s_m\}$ diagonalize the bilinear form $b(v,w):=\int_{B_{\sigma/8}(0)}((x-\bar{x})\cdot v)((x-\bar{x})\cdot w) \, d\mu,$ with corresponding eigenvalues $0 \leq \lambda_m \leq ... \leq \lambda_1.$ Then the definition of barycenter and \cite[(6.1)]{dees} imply that for each $s_j$ and every $y \in B_{\sigma/4}(0)$
    \[\nabla v(y) \circ (-\lambda_js_j)=\int_{B_{\sigma/8}(0)}((x-\bar{x})\cdot s_j)(\nabla v\circ(y-x)-\alpha v(y)) \, d\mu(x). \]
    Square both sides and apply \cite[(6.1)]{dees} and definition of barycenter: For some constant $\alpha>0,$
    \begin{align*}
        \lambda_j^2|\partial_{s_j}v|^2 (y)&\leq \left(\int_{B_{\sigma/8}(0)}\big|(x-\bar{x})\cdot s_j\big| \left|\nabla v(y)\circ(y-x)-\alpha v(y)\right| \, d\mu(x) \right)^2\\
        &\leq \int_{B_{\sigma/8}(0)} ((z-\bar{x})\cdot s_j)^2 \, d\mu(z) \ \ \cdot \ \ \int_{B_{\sigma/8}(0)} \left|\nabla v(y)\circ(y-x)-\alpha v(y)\right|^2 d\mu(x)\\
        &= \lambda_j \int_{B_{\sigma/8}(0)} \left|\nabla v(y)\circ(y-x)-\alpha v(y)\right|^2 \, d\mu(x).
    \end{align*}
    \begin{align*}
        \lambda_j|\partial_{s_j}v|^2(y)&\leq \int_{B_{
        \sigma/8}(0)} \left|\nabla v(y)\circ(y-x)-\alpha v(y)\right|^2 \, d\mu(x).
    \end{align*}
    Sum for $j=1$ to $k+1$ and integrate over $y \in B_{5\sigma/4}(0)\setminus B_{3\sigma/4}(0),$
    {\small\begin{align*}
        D_{\mu}^k(0,\sigma/8)&\int_{B_{5\sigma/4}(0)\setminus B_{3\sigma/4}(0)}\sum^{k+1}_{j=1} \left| \partial_{s_j} v(y) \right|^2 \, dy =\left(\frac{\sigma}{8}\right)^{-k-2} \int_{B_{5\sigma/4}(0)\setminus B_{3\sigma/4}(0)} \paren{\sum_{i=k+1}^n\lambda_i}\sum^{k+1}_{j=1} \left|\partial_{s_j} v(y)\right|^2 \, dy\\
        &\leq C\sigma^{-k-2}\int_{B_{5\sigma/4}(0) \setminus B_{3\sigma/4}(0)}\lambda_{k+1}\sum^{k+1}_{j=1} \left| \partial_{s_j} v(y)\right|^2 \, dy\\
        &\leq C\sigma^{-k-2}\int_{B_{5\sigma/4}(0) \setminus B_{3\sigma/4}(0)}\sum^{k+1}_{j=1} \lambda_j \left| \partial_{s_j} v(y)\right|^2 \, dy\\  
        &\leq C\sigma^{-k-2} \int_{B_{5\sigma/4}(0)\setminus B_{3\sigma/4}(0)} \int_{B_{\sigma/8}(0)} \left|\nabla v(y)\circ(y-x)-\alpha v(y)\right|^2 \, d\mu(x) \, dy,
    \end{align*}}
    where $C$ denotes a generic constant depending only on $(n,\rho,\eta,M,Y_C,\Lambda)$ for simplicity.

    Then, Lemma \ref{toolformainflatness} gives us
    \begin{align*}
        D^{k}_{\mu}(0,\frac{\sigma}{8}) &\leq \ C \sigma^{-k-n}\int_{B_{\sigma/8}(0)} \int_{B_{3\sigma/2}(x) \setminus B_{\sigma/2}(x)}\left|\nabla v(y)\circ(y-x)-\alpha v(y)\right|^2 \, dy \, d\mu(x)\\
        &\leq C\sigma^{-k-n}\int_{B_{\sigma/8}(0)}\int_{B_{3\sigma/2}(x) \setminus B_{\sigma/2}(x)} \left| \nabla v(y)\circ(y-x)-\BDOrd_{\phi,v}(x,\sigma)v(y) \right|^2 \, dy \, d\mu(x)\\
        &+ C \sigma^{-k-n}\int_{B_{\sigma/8}(0)} \int_{B_{3\sigma/2}(x)\setminus B_{\sigma/2}(x)}(\BDOrd_{\phi,v}(x,\sigma)-\alpha)^2|v(y)|^2 \, dy \, d\mu(x)\\
        &\leq C\sigma^{-k-n}\int_{B_{\sigma/8}(0)}\int_{B_{3\sigma/2}(x) \setminus B_{\sigma/2}(x)} \left| \nabla v(y)\circ(y-x)-\BDOrd_{\phi,v}(x,\sigma)v(y) \right|^2 \, dy \, d\mu(x)\\
        &+ C\sigma^{-k-n} \int_{B_{\sigma/8}(0)} \int_{B_{3\sigma/2}(x)}(\BDOrd_{\phi,v}(x,\sigma)-\alpha)^2|v(y)|^2 \, dy \, d\mu(x)\\
        =&: \ I + II.
    \end{align*}
    Let $\alpha:=\frac{1}{\mu(B_{\sigma/8}(0))} \int_{B_{\sigma/8}(0)}\BDOrd_{\phi,v}(y,\sigma) \, d\mu(y).$ Assume that $\sigma<R_1$. Since $\int_{B_{3\sigma/2}(x)}|v(y)|^2 dy$ can be bounded by $C\sigma I_{\phi,v}(x,3\sigma)$ for some constant $C$ by \eqref{intdbound}, which can then be bounded by $C\sigma I_{\phi,v}(0,12\sigma)$ by Lemma \ref{monotoneheight}, which can then be bounded by $C\sigma I_{\phi,v}(0,\sigma)$ by \eqref{eqn:6}, we have from \eqref{eq:pinchingbound},
    \begin{align*}
        II &\leq C\sigma^{-k+1-n}I_{\phi,v}(0,\sigma) \int_{B_{\sigma/8}(0)}\left(\BDOrd_{\phi,v}(x,\sigma)-\frac{1}{\mu(B_{\sigma/8}(0))}\int_{B_{\sigma/8}(0)}\BDOrd_{\phi,v}(y,\sigma) \, d\mu(y)\right)^2 \, d\mu(x)\\
        &\leq \frac{C\sigma^{-k+1-n}I_{\phi,v}(0,\sigma)}{\mu(B_{\sigma/8}(0))} \int_{B_{\sigma/8}(0)} \int_{B_{\sigma/8}(0)} \left( \BDOrd_{\phi,v}(x,\sigma)-\BDOrd_{\phi,v}(y,\sigma)\right)^2 \, d\mu(y) \, d\mu(x)\\
        &\leq \frac{C\sigma^{-k+1-n}I_{\phi,v}(0,\sigma)}{\mu(B_{\sigma/8}(0))} \int_{B_{\sigma/8}(0)} \int_{B_{\sigma/8}(0)} \left( W(x)+W(y)+\sigma \right) \, d\mu(y) \, d\mu(x)\\
        &= C\sigma^{-k+1-n}I_{\phi,v}(0,\sigma) \left( 2\int_{B_{\sigma/8}(0)} W(x) d\mu(x)+\sigma \right).
    \end{align*}
    By the triangle inequality,
    \begin{align*}
     I \leq& \ C \sigma^{-k-n}\int_{B_{\sigma/8}(0)} \int_{B_{3\sigma/2}(x) \setminus B_{\sigma/2}(x)} \left| \BDOrd_{\phi,v}(x,\sigma)-\BDOrd_{\phi,v}(x,|y-x|)\right|^2 |v(y)|^2 \, dy \, d\mu(x)\\
     &+C \sigma^{-k-n}\int_{B_{\sigma/8}(0)} \int_{B_{3\sigma/2}(x)\setminus B_{\sigma/2}(x)} \left| \nabla v(y)\circ(y-x)-\BDOrd_{\phi,v}(x,|y-x|)v(y)\right|^2 \, dy \, d\mu(x)\\
     =& \ I_1+I_2.
    \end{align*}
    As $r \mapsto \Ord_{\phi,v}(x,r)$ is nondecreasing and $\Ord_{\phi,v}(0,R_0)\leq \Lambda,$ it follows that for $x \in B_{\sigma/8}(0)$ and $\frac{\sigma}{2}<|y-x|<\frac{3\sigma}{2},$
    \begin{align*}
    |\BDOrd_{\phi,v}(x,\sigma)-\BDOrd_{\phi,v}(x,|y-x|)| &\leq |W_{|y|}^{\sigma}(x,v)|+A(\sigma^2-|y|^2)+A(\sigma-|y|)\Ord_{\phi,v}(x,|y-x|)\\
    &\leq W_{\sigma/2}^{3\sigma/2}(x,v)+A(1-1/64)\sigma^2+A(1-1/8)\sigma\Ord_{\phi,v}(x,\sigma)\\
    & \leq C \,( W_{\sigma/8}^{4\sigma}(x,v)+\sigma).
    \end{align*}
    Thus,
    \begin{align*}
    I_1 \leq& \ C\sigma^{-k-n} \int_{B_{\sigma/8}(0)} \int_{B_{3\sigma/2}(x)\setminus B_{\sigma/2}(x)} |W^{4\sigma}_{\sigma/8}(x)+\sigma|^2 \left|v(y)\right|^2 \, dy d\mu(x) \\
    \leq & \ C \sigma^{-k+1-n}I_{\phi,v}(0,\sigma)\int_{B_{\sigma/8}(0)} |W^{4\sigma}_{\sigma/8}(x)+\sigma|^2 \, d\mu(x)\\
    \leq & \ C\sigma^{-k+1-n}I_{\phi,v}(0,\sigma)\paren{\int_{B_{\sigma/8}(0)} W^{4\sigma}_{\sigma/8}(x) \, d\mu(x)+\sigma}.
    \end{align*}
    Let $R=3\sigma$ and $\theta'=\frac{1}{6}.$ By the argument of Lemma \ref{lemma:d5.3}, \eqref{eqn:6}, Lemma \ref{monotoneheight}, and 
    Lemma \ref{monotoneorder},
    \begin{align*}
        I_2 \leq& \ C \sigma^{-n-k}\int_{B_{\sigma/8}(0)} 3\sigma I_{\phi,v}(x,3\sigma)(W^{3\sigma}_{\sigma/4}(x)+(3\sigma)^2) \, d\mu(x) \\
        \leq& \ C \sigma^{-k+1-n}I_{\phi,v}(0,\sigma)\paren{\int_{B_{\sigma/8}(0)}W^{4\sigma}_{\sigma/8}(x)\, d\mu(x)+\sigma}.
    \end{align*}
    Combining together, we have
    \[
    D^{m-2}_{\mu}(0,\frac{\sigma}{8}) \leq C\sigma^{-k+1-n}I_{\phi,v}(0,\sigma)\paren{ \int_{B_{\sigma/8}(0)} W^{4\sigma}_{\sigma/8}(x) \, d\mu(x)+\sigma}.
    \]
\end{proof}

\section{Main Results}\label{sec:mainresults}
\subsection{Local Results}We will give the local statements with respect to the conical DM-complexes by using the covering theorem in Appendix.
\begin{theorem}\label{bd:4.1}
    Let $u$ be a harmonic map to an $N$-dimensional conical DM-complex $Y_C$ with splitting $u=(V,v):B_{\sigma_j}(0) \to (\mathbb{R}^j \times Y_C^{N-j},d_G)$ satisfying $\Ord_{\phi,v}(0,R_0) \leq \Lambda$ and $\sup_{B_{\sigma_j}(0)}|\nabla u| \leq M.$ Then for any $0<\eta<\eta_1,$ where $\eta_1$ is defined as in Lemma~\ref{bd5.12}, there are constants $\sigma, C=(\sigma,C)(n,\eta,Y_C,M,\Lambda)>0,$ such that for any $0<r<\sigma/8$, we have the Minkowski-type estimate
    \[
    \left|B_r(\calS_{0,\eta}^k(v)\cap B_{\sigma/8}(0))\right|\leq C r^{n-k}.
    \]
\end{theorem}
\begin{proof}
    Our approach is analogous to the proof of \cite[Theorem 4.1]{bd}. Fix any $0 <\eta <\eta_1(n)$ and $\sigma \in (0,\sigma_{\ref{bd:7.1}}]$ where $\eta_1(n)$ and $\sigma_{\ref{bd:7.1}}$ are chosen so that Theorem \ref{bd:7.1} applies. Denote $D_0=\calS_{0,\eta}^k(v) \cap B_{\sigma/8}(0).$ Lemma \ref{ineq:ord} implies that for any $y \in B_{\sigma/8}(0), \ \Gamma:=\sup_{y \in D_0} \Ord_{\phi,v}(y,\sigma)\leq \sup_{y \in D_0} C(\Ord_{\phi,v}(0,16\sigma)+1)\leq C(\Lambda).$ \ We want to prove the Minkowski estimate on $B_r(D_0)$ by repeatedly applying the covering theorem at smaller and smaller scales.

    The initial step is to apply Theorem \ref{bd:7.1} by letting $S=\frac{\sigma}{8}, \, s=r, D=D_0$ and $\{B_{s_x}(x)\}_{x\in\calC(1)}$ be the corresponding cover with $D_x \subseteq D_0.$ It follows that
    \[
    \sum_{x \in \calC(1)}(s_x)^k \leq C.
    \]
    If $s_x=r,$ we stop refining $B_{s_x}(x)$ and define $x \in \mathcal{G}(1):=\{x:s_x=r\}$ to be the good points. Otherwise, define $\calB(1):=\calC(1) \setminus \mathcal{G}(1)$ to be the set of bad points with
    \[
    \Gamma_x:=\sup_{y \in D_x} \Ord_{\phi,v}(y,8s_x)\leq \Gamma-\delta.
    \]
    We further refine $\calB(1)$ by applying Theorem \ref{bd:7.1}. For each $x \in \calB(1),$ let $S=s_x,s=r,$ and $D=D_x.$ We have the refined covering $\{B_{s_y}(y)\}_{y\in\calC(2)_x}$ and the corresponding decomposition of $D_x$ into $D_{x,y}=:D_y.$ Theorem \ref{bd:7.1} implies that if $s_y \neq r,$ then
    \[
    \sup_{z \in D_{x,y}} \Ord_{\phi,v}(z,8s_y) \leq \Gamma_x-\delta \leq \Gamma-2\delta.
    \]
    Now we have the covering $\calC(2):=\mathcal{G}(1) \bigcup \cup_{x \in \calB(1)}\calC(2)_x.$ By Theorem \ref{bd:7.1}, we have
    \[
    \sum_{x \in \calC(2)}(s_x)^k \leq C\sum_{x \in \calC(1)}(s_x)^k \leq C^2.
    \]
    Again, we define $\mathcal{G}(2):=\{x \in \calC(2):s_x=r\}$ and $\calB(2):=\calC(2)\setminus\mathcal{G}(2).$ By repeating the same procedure inductively, we have a covering $\{B_{s_x}(x)\}_{x \in\calC(m)}$ such that
    \[
    \sum_{x \in \calC(m)}(s_x)^k \leq C^m
    \]
    and for each $x \in \calC(m),$ either $s_x=r$ or
    \[
    \sup_{y \in D_x}\Ord_{\phi,v}(y,8s_x) \leq \Gamma-m\delta.
    \]
    We want to get $m\geq \lceil \frac{\Gamma}{\delta} \rceil,$ which is the number of steps only depending on $n,\eta,Y_C,\Lambda.$ At the least such $m,$ we have the covering $\{B_r(x)\}_{x \in\calC(m)}$ of $D_0$ such that
    \[
    \sum_{x \in \calC(m)} r^k \leq C^m.
    \]
    Because $D_0 \subseteq \cup_{x \in \calC(m)}B_r(x),$ $B_r(D_0) \subseteq \cup_{x \in \calC(m)} B_{2r}(x).$ Then we have the bound of the volume of $B_r(D_0)$
    \[
    |B_r(D_0)| \leq \sum_{x \in \calC(m)} \omega_n2^nr^n=\omega_n2^nr^{n-k}\sum_{x \in \calC(m)} r^k \leq \omega_n 2^n C^m r^{n-k}.
    \]
    Let $C'=\omega_n 2^n C^m=C'(n,\eta,Y_C,\Lambda).$ Thus we have the Minkowski bound.
\end{proof}
\begin{theorem}\label{bd:4.2}
Let $u$ be a harmonic map to an $N$-dimensional conical DM-complex $Y_C$ with splitting $u=(V,v):B_{\sigma_j}(0) \to (\mathbb{R}^j \times Y_C^{N-j},d_G)$ satisfying $\Ord_{\phi,v}(0,R_0) \leq \Lambda,$ and  $\sup_{B_{\sigma_j}(0)}|\nabla u| \leq M.$ There exists $\eta_1>0$ as in Lemma \ref{bd5.12} so that if  $0<\eta<\eta_1,$ then $\calS_{0,\eta}^k(v) \cap B_{\sigma_{\ref{bd:4.1}}}(0)$ is $k$-rectifiable.
\end{theorem}
\begin{proof}
    Choose $\eta_1$ as in Lemma \ref{bd5.12} so that we can invoke Theorem \ref{bd:4.1}. Set \\ $S_i:=\{x\in \calS_{0,\eta}^k:\Ord_{\phi,v}(x,32R_0/i)-\Ord_{\phi,v}(x,0)<\delta\}$ and $\mu_i:=\calH^k\resmes S_i$. For $i$ sufficiently large, we have $32R_0/i<\sigma_{ \ref{bd:4.1}}$, where $\sigma_{\ref{bd:4.1}}$ is the $\sigma$ of Theorem \ref{bd:4.1}. Then, for any $s$ with $s\leq R_0/i$ and any $x\in S_i$, we have $W_s^{32s}(x)<\delta$. Then, for any $B_{2t}(y)\subset B_{\sigma_{\ref{bd:4.1}}}(0)$, we have as in the proof of \cite[Theorem 4.2]{bd} that
    \begin{align*}
        \int_{B_t(y)}\int_0^tD_{\mu_i}^k(z,s)\frac{ds}{s}d\mu_i(z)&\leq C\int_{B_{2t}(y)}\int_0^t(W_s^{32s}(\xi)+s)\frac{ds}{s}d\mu_i(\xi)\\
        &\leq C\paren{t\mu_i(B_{2t}(y))+\int_{B_{2t}(y)}\int_0^tW_s^{32s}(\xi)\frac{ds}{s}d\mu_i(\xi)}
    \end{align*}
    Since $\Ord_{\phi,v}(\xi,s)$ is monotone nondecreasing in $s$, we can show by a method similar to Equation \cite[(8.10)]{dees} that $\int_0^tW_s^{32s}(\xi)\frac{ds}{s}$ is finite and bounded uniformly in $\xi$. Thus, we have $\int_{B_t(y)}\int_0^tD_{\mu_i}^k(z,s)\frac{ds}{s}d\mu_i(z)<\infty$, so $\int_0^tD_{\mu_i}^k(z,s)\frac{ds}{s}<\infty$ for $\mu_i$-a.e. $z\in B_t(y)$, which implies by Lemma \ref{rectmeanflatness} that $S_i\cap B_t(y)$ is $k$-rectifiable. Taking a countable cover of $B_{\sigma_{\ref{bd:4.1}}}(0)$ by such balls, we see that $S_i\cap B_{\sigma_{\ref{bd:4.1}}}(0)$ is rectifiable. Since $\{S_i\cap B_{\sigma_{\ref{bd:4.1}}}(0)\}_{i=1}^{\infty}$ form a countable cover of $\calS_{0,\eta}^k(v)\cap B_{\sigma_{\ref{bd:4.1}}}(0)$, the latter is rectifiable. 
\end{proof}
Using the previous results above, we are ready to give the $k$-rectifiability for $\calS_0^k(v) \cap B_{R_0/64}(0).$
\begin{corollary}\label{s0krect}
    In the setting of the previous lemma, $\calS_0^k(v)\cap B_{R_0/64}(0)$ is countably $k$-rectifiable.
\end{corollary}
\begin{proof}
    By Lemma \ref{ineq:ord}, we have for any $x\in B_{R_0/64}(0)$ that $\Ord_{\phi,v}(x,R_0/16)\leq C(\Lambda+1):=\Lambda'$ for all $x\in B_{R_0/64}(0)$. Define $\eta_2=\eta_1(n,\sigma_j-R_0/64,Y_C,\Lambda',M)$, and for any $0<\eta<\eta_2$, let $\sigma_{\eta}=\sigma_{\ref{bd:4.1}}(\eta,n,\sigma_j-R_0/64,M,\Lambda')$. Then, applying Theorem \ref{bd:4.2}, for any $0<\eta\leq \eta_2$ and any $x\in S_{0,\eta}^k(v)$, the set $S_{0,\eta}^k(u)\cap B_{\sigma_{\eta}}(x)$ is $k$-rectifiable. Covering $S_{0,\eta}^k(v)\cap B_{R_0/64}(0)$ by countably many such balls, we see that $S_{0,\eta}^k(v)\cap B_{R_0/64}(0)$ is countably $k$-rectifiable. It then follows from Lemma \ref{bd3.7} that $\calS_0^k(v)\cap B_{R_0/64}(0)=B_{R_0/64}(0)\cap \cup_{i=1}^{\infty}\calS_{0,1/i}^k(v)$ is countably $k$-rectifiable.
\end{proof}
\subsection{Proof of Main Theorem} Consequently, we prove Theorem \ref{thm:mainrect} by applying Corollary \ref{s0krect}.
\begin{proof}[Proof of Theorem \ref{thm:mainrect}]
    It clearly suffices to show that $\calS_j^k(u)$ is $k$-rectifiable for each $j\in\{0,1,\ldots,N-1\}$. Fix $j\in \{0,1,\ldots,N-1\}$ and let $x\in \calS_j^k(u)$. Consider splitting data $(\sigma_j(x),V_x,v_x,j,Y_x)$ for $u$ at $x$. It follows from Corollary \ref{s0krect} that there exists $R(x)>0$ such that $\calS_0^k(v_x)\cap B_{R(x)}(x)$ is $k$-rectifiable. It follows from the definitions that $\calS_j^k(u)\cap B_{R(x)}(x)\subset \calS_0^k(v_x)\cap B_{R(x)}(x)$, since for any $y\in \calS_j^k(u)\cap B_{R(x)}(x)$, the splitting data $(\sigma_j(x)-|x-y|,V_x,v_x,j,Y_x)$ has optimal dimension, and so $\sigma_j(x)-|x-y|\leq \sigma_j(y)$, which implies $y\in \calS_0^k(v_x)$. Covering $\calS_j^k(u)$ by countably many balls $B_{\sigma_j(x)}(x)$, since each $\calS_0^k(v_x)\cap B_{R(x)}(x)$ is $k$-rectifiable, so is $\calS_j^k(u)$. It follows that $\calS^k(u)$ is $k$-rectifiable.
\end{proof}
\section{Appendix}\label{sec:appendix}
The covering arguments below basically follow the proofs in \cite{bd} with further modifications with respect to the local splitting data of $u=(V,v)$ and the scales derived from the previous lemmas.
\begin{theorem}[Theorem 3.4 in \cite{nv1}]\label{bd:A1}
There exist $\delta_R(n)>0$ and $C_R(n)>0$ such that the following argument holds.
For any $R>0$, let $\{B_{r_j/5}(x_j)\}_{x_j\in S}\subset B_{3R}(0)$ be mutually disjoint
with $x_j\in B_R(0)$, and let
\[
\mu=\sum_j \omega_k r_j^k\delta_{x_j}.
\]
If for every $B_r(y)\subset B_{2R}(0)$ we have
\[
\int_{B_r(y)}\left(\int_0^r D_\mu^k(z,s)\frac{ds}{s}\right)\,d\mu(z) \leq \delta_R^2 r^k,
\]
then
\[
\mu(B_R(0))=\omega_k\sum_{i \in I}r_i^k\le C_R(n)R^k.
\]
\end{theorem}
\begin{proposition}\label{bd:A2}
    For a harmonic map $u$ into DM-complex $X,$ let $u=(V,v):B_{\sigma_j}(0)\to \mathbb{R}^j\times Y_C^{N-j}$ be a local representation of a harmonic map, where $Y_C^{N-j}$ is a conical $F$-connected complex satisfying $\Ord_{\phi,v}(0,R_0)\le \Lambda \text{ and } \sup_{B_{\sigma_j}(0)}|\nabla u| \leq M.$ For $0<\rho\le \frac1{256}$ and every $0<\eta<\eta_1$, where $\eta_1$ is chosen as in Lemma \ref{bd5.12}, there exist $\delta=\delta(n,\rho,\eta,Y_C, \Lambda)>0, C=C(n) \text{ and } \sigma\in(0,\min\{\frac{\sigma_j}{16},\frac{4}{5\rho},
\sigma_{\ref{bd5.8}, \ref{bd5.10}, \ref{bd5.13}, \ref{bd5.14}, \ref{toolformainflatness}, \ref{bd6.2}}\}]$ so that the following holds. 

For any $x \in B_{\sigma/8}(0),$ let $s,S$ satisfy $0<s<S \leq \frac{\sigma}{8}, \, D \subset \calS^k_{0,\eta,\delta s}(v)\cap B_S(x),$
and $\Gamma:=\sup_{y\in D}\Ord_{\phi,v}(y,8S).$ Then there is a covering $\{B_{r_i}(x_i)\}_{i\in I}$ of $D$ by finitely many balls such that
\begin{enumerate}
\item [(1)] $r_i\ge 10\rho s;$
\item [(2)] $\sum_i r_i^k \leq CS^k;$
\item [(3)] for each $i$, either $r_i\le s$, or the set
\[
F_i:=D\cap B_{r_i}(x_i)\cap
\{y:\Ord_{\phi,v}(y,\rho r_i)>\Gamma-\delta\}
\]
is contained in $B_{r_i}(x_i)\cap B_{\rho r_i}(L_i)$ for a $(k-1)$-dimensional affine subspace $L_i$.
\end{enumerate}
\end{proposition}
\begin{proof}
    The proof is analogous to the proof of \cite[Proposition A.2]{bd}.
\end{proof}

\begin{theorem}\label{bd:7.1}
    Let $u=(V,v):B_{\sigma_j}(0) \to \mathbb{R}^j \times Y_C^{N-j}$ be a local representation of a harmonic map into a conical DM-complex $Y_C^N$ satisfying $\Ord_{\phi,v}(0,R_0)\leq \Lambda$ and $\sup_{B_{\sigma_j}(0)}|\nabla u| \leq M.$ Given $k\in\{0,...,n-1\}$ and $0<\eta<\eta_1$ where $\eta_1$ is chosen as in Lemma \ref{bd5.12}, there exist $\delta=\delta(n,\eta,Y_C,\Lambda)>0, \, C=C(n)>0,$ and $\sigma\in (0,\sigma_{\ref{bd:A2}}]$ such that the following holds.

    Let $0<s<S<\frac{\sigma}{8}$ and $x \in B_{\sigma}(0)$. Suppose that $D \subseteq \calS_{0,\eta,\delta s}^k \cap B_S(x)$ is any subset and $\Gamma:=\Ord_{\phi,v}(y,8S).$ Then there is a finite decomposition of $D$ into sets $D_x,$ a finite set $\calC \subseteq B_{2\sigma}(0),$ and a collection of balls $\{B_{s_x}(x)\}_{x \in \calC}$ with $s_x \geq s,$ such that:
    \begin{itemize}
        \item [(1)]$D_x \subseteq B_{s_x}(x),$
        \item [(2)]$\sum s_x^k \leq CS^k,$
        \item [(3)]for each $x \in \calC,$ either $s_x=s$ or for any $y \in D_x,$
        \begin{equation*}
            \Ord_{\phi,v}(y,8s_x) \leq \Gamma-\delta.
        \end{equation*}
    \end{itemize}
\end{theorem}
\begin{proof}
    The proof follows arguments similar to the proof of \cite[Theorem 7.1]{bd}.
\end{proof}
\bibliographystyle{amsalpha}
\bibliography{rectifiability}
\end{document}